\documentclass[a4paper,12pt,reqno,english]{amsart}
\usepackage[utf8]{inputenc}
\usepackage[T1]{fontenc}
\usepackage{babel}

\usepackage[bookmarks = true, colorlinks, citecolor = blue, urlcolor = blue]{hyperref}
\usepackage[capitalize, nameinlink]{cleveref}

\numberwithin{equation}{section}
\numberwithin{figure}{section}

\theoremstyle{plain}
\newtheorem{theorem}{Theorem}[section]

\theoremstyle{plain}
\newtheorem*{theorem*}{Theorem}

\theoremstyle{plain}
\newtheorem{proposition}[theorem]{Proposition}

\theoremstyle{plain}
\newtheorem{lemma}[theorem]{Lemma}

\theoremstyle{plain}
\newtheorem{corollary}[theorem]{Corollary}

\theoremstyle{definition}
\newtheorem{definition}[theorem]{Definition}

\theoremstyle{definition}
\newtheorem{notation}[theorem]{Notation}

\theoremstyle{definition}
\newtheorem{example}[theorem]{Example}

\theoremstyle{remark}
\newtheorem{remark}[theorem]{Remark}

\addto\extrasenglish{
  
}

\newcommand{\Uqg}{U_q(\mathfrak{g})}
\newcommand{\Uqlevi}{U_q(\mathfrak{l}_S)}
\newcommand{\Cqg}{\mathbb{C}_q[G]}
\newcommand{\flagL}{\mathbb{C}_q[G / T]}
\newcommand{\flagR}{\mathbb{C}_q[T \backslash G]}
\newcommand{\flagleviL}{\mathbb{C}_q[G / L_S]}
\newcommand{\flagleviR}{\mathbb{C}_q[L_S \backslash G]}

\newcommand{\lieg}{\mathfrak{g}}

\newcommand{\CqG}{\mathbb{C}_q[G]}

\begin{document}

\title[Twisted 2-cycles on quantum flag manifolds]{Twisted Hochschild homology of quantum flag manifolds: 2-cycles from invariant projections}

\author{Marco Matassa}%

\email{marco.matassa@oslomet.no}

\address{OsloMet – storbyuniversitetet}

\begin{abstract}
We study the twisted Hochschild homology of quantum full flag manifolds, with the twist being the modular automorphism of the Haar state.
We show that non-trivial $2$-cycles can be constructed from appropriate invariant projections.
Moreover we show that $HH_2^\theta(\flagL)$ has dimension at least $\mathrm{rank}(\lieg)$.
We also discuss the case of generalized flag manifolds and present the example of the quantum Grassmannians.
\end{abstract}

\maketitle

\section*{Introduction}

In this article we will study some aspects of the twisted Hochschild homology of certain quantized coordinate rings.
These rings, which will be denoted by $\flagL$, are quantizations of the coordinate rings of the full flag manifolds $G / T$.
They will be defined starting from the quantized coordinate rings of the corresponding compact Lie groups, denoted by $\Cqg$.
Several of the results which we are going to prove will hold in this setting as well.
We will focus on the degree-two part where, as we will show, it is possible to produce many non-trivial classes from appropriate invariant projections.
Below we will discuss some reasons why we consider the degree-two part to be very interesting.
One of our main result is the following, where we write $\lieg$ for the Lie algebra of $G$ and $\mathrm{rank}(\lieg)$ for its rank.

\begin{theorem*}
Let $\theta$ be the modular automorphism of the Haar state of $\Cqg$.
Then the twisted Hochschild homology group $HH_2^\theta(\flagL)$ has dimension at least $\mathrm{rank}(\lieg)$.
\end{theorem*}

The case $\mathrm{rank}(\lieg) = 1$, corresponding geometrically to the quantum $2$-sphere, was previously known \cite{twisted-s2}.
In this case we have that $HH_2^\theta(\flagL)$ is $1$-dimensional.

Let us briefly provide some background and motivations for this article.
First we need to recall the well-known \emph{Hochschild-Kostant-Rosenberg} theorem: given a smooth variety $X$ and its coordinate ring $A$, the Hochschild homology $HH_\bullet(A)$, which is a graded-commutative algebra via the shuffle product, can be identified with the algebra of differential forms on $X$.
For non-commutative algebras, which we consider as quantum spaces, this motivates taking the Hochschild homology as a replacement for the differential forms, a point of view which is advocated in Connes' approach to non-commutative geometry \cite{connes}.
However the Hochschild homology of quantum spaces tends to be fairly degenerate compared to their classical counterparts, a phenomenon usually referred to as the "dimension drop".

On the other hand, the situation changes upon introducing some twisting in the coefficients.
This setting was introduced in \cite{kmt} for compact quantum groups, with the aim of finding a connection with Woronowicz's theory of covariant differential calculi \cite{woronowicz}.
Concrete computations were performed in \cite{twisted-s2, twisted-su2, sln, sl2-second}, showing that indeed appropriate twisting avoids the "dimension drop".
A more conceptual understanding of this phenomenon was given in \cite{brzh}, where it was connected with a general version of Poincaré duality for certain non-commutative algebras, known as \emph{Van den Bergh duality}.

Here we will focus on the study of twisted $2$-cycles on quantum full flag manifolds.
As we have mentioned above, in this case it is possible to produce many non-trivial classes from appropriate invariant projections.
This is interesting because the general results that are available do not give easy access to the intermediate degrees.
Another important motivation is that among $2$-cycles we expect to find examples of quantum Kähler forms, since the classical manifolds we are considering are Kähler.
We will come back to this point in the last section, where we will discuss the concrete example of the quantum Grassmannians.

The paper is organized as follows.
In \cref{sec:notation} we provide some background and fix notations and conventions.
In \cref{sec:hochschild} we recall basic facts on Hochschild homology, review known results on quantized coordinate rings and prove a simple result regarding twisted $2$-cycles.
In \cref{sec:projections} we define projections on quantized coordinate rings using appropriate matrix units.
In \cref{sec:flag} we show how these projections are connected to quantum flag manifolds and equivariant $K$-theory.
In \cref{sec:twisted} we show that these projections can be used to define twisted $2$-cycles. We also introduce some $2$-cocycles, in order to prove their non-triviality.
In \cref{sec:pairing} we compute the pairings of the cycles with the cocycles.
In \cref{sec:classes} we discuss non-triviality and linear independence of these classes, as well as proving our main theorem.
Finally in \cref{sec:generalized} we extend some of the previous results to generalized flag manifolds.
In particular we present the example of the quantum Grassmannians.

\section{Notations and conventions}
\label{sec:notation}

In this section we fix some basic notation and briefly review some facts about complex simple Lie algebras, quantized enveloping algebras and quantized coordinate rings.

\subsection{Quantized enveloping algebras}

Let $\mathfrak{g}$ be a finite-dimensional complex simple Lie algebra with fixed Cartan subalgebra $\mathfrak{h}$.
We denote by $\Delta(\mathfrak{g})$ the root system, by $\Delta^{+}(\mathfrak{g})$ the positive roots and by $\Pi = \{ \alpha_{1}, \cdots, \alpha_{r} \}$ the simple roots.
The Killing form induces an invariant bilinear form on $\mathfrak{h}^*$, normalized so that for every short root $\alpha_i$ we have $(\alpha_i, \alpha_i) = 2$.
The Cartan matrix $(a_{ij})$ is then defined by $(\alpha_i, \alpha_j) = d_i a_{ij}$, where $d_i := (\alpha_i, \alpha_i) / 2$.

For quantized enveloping algebras we use the conventions of \cite{klsc}.
Let $q \in \mathbb{C}$ be non-zero and define $q_i := q^{d_i}$. Suppose that $q_i^2 \neq 1$ for all $i$.
The \emph{quantized universal enveloping algebra} $U_q(\mathfrak{g})$ is generated by $\{E_i$, $F_i$, $K_i$, $K_i^{-1}\}_{i = 1}^r$, where $r$ is the rank of $\mathfrak{g}$, with relations
\begin{gather*}
K_i K_i^{-1} = K_i^{-1} K_i=1,\ \
K_i K_j = K_j K_i, \\
K_i E_j K_i^{-1} = q_i^{a_{ij}} E_j,\ \
K_i F_j K_i^{-1} = q_i^{-a_{ij}} F_j, \\
E_i F_j - F_j E_i = \delta_{ij} \frac{K_i - K_i^{-1}}{q_i - q_i^{-1}},
\end{gather*}
plus the quantum analogue of the Serre relations.
The Hopf algebra structure is defined by
\begin{gather*}
\Delta(K_i) = K_i \otimes K_i, \quad
\Delta(E_i) = E_i \otimes K_i + 1 \otimes E_i, \quad
\Delta(F_i) = F_i \otimes 1 + K_i^{-1} \otimes F_i, \\
S(K_i) = K_i^{-1}, \quad
S(E_i) = - E_i K_i^{-1}, \quad
S(F_i) = - K_i F_i, \quad
\varepsilon(K_i) = 1, \quad
\varepsilon(E_i) = \varepsilon(F_i) = 0.
\end{gather*}

For $\lambda = \sum_{i = 1}^r n_i \alpha_i$ we will write $K_\lambda = K_1^{n_1} \cdots K_r^{n_r}$.
Let $\rho$ be the half-sum of the positive roots of $\mathfrak{g}$. Then we have $S^2(X) = K_{2\rho} X K_{2\rho}^{-1}$ for any $X \in \Uqg$.
For $q \in \mathbb{R}$ we can define the \emph{compact real form} of $U_q(\mathfrak{g})$, which makes it into a Hopf $*$-algebra. It is defined by
\[
K_i^* = K_i, \quad
E_i^* = K_i F_i, \quad
F_i^* = E_i K_i^{-1}.
\]

\subsection{Quantized coordinate rings}

Dually to the quantized enveloping algebra $\Uqg$ we define the \emph{quantized coordinate ring} $\Cqg$, whose elements should be interpreted as "functions" on the corresponding compact quantum group. We define $\Cqg$ as the subspace of the linear dual $\Uqg^*$ spanned by the matrix coefficients of finite-dimensional representations of $\Uqg$.
The Hopf $*$-algebra structure of $\Uqg$ induces a Hopf $*$-algebra on $\Cqg$ by the formulae
\begin{gather*}
(\phi \psi)(X) = (\phi \otimes \psi) \Delta(X), \quad
1(X) = \varepsilon(X), \\
\Delta(\phi)(X \otimes Y) = \phi(X Y), \quad
\varepsilon(\phi) = \phi(1), \\
S(\phi)(X) = \phi(S(X)), \quad
\phi^*(X) = \overline{\phi(S(X)^*)}.
\end{gather*}
Here $\phi, \psi \in \Cqg$ and $X, Y \in \Uqg$.
More precisely, given an irreducible representation $V(\Lambda)$ of highest weight $\Lambda$, the matrix coefficients are defined by
\[
c^{\Lambda}_{f, v}(X) = f(X \triangleright v), \quad
v \in V(\Lambda), \ f \in V(\Lambda)^*, \ X \in \Uqg.
\]
The quantized coordinate ring $\Cqg$ is a $\Uqg$-bimodule in a natural way via
\[
(X \triangleright \phi)(Y) = \phi(Y X), \quad
(\phi \triangleleft X)(Y) = \phi(X Y).
\]
It is well known that the finite-dimensional irreducible representations $V(\Lambda)$ are unitarizable.
Therefore we are free to choose an orthonormal basis $\{v_i\}_i$ of $V(\Lambda)$. It will be convenient to do so in the following.
We also have a corresponding dual basis $\{f^i\}_i$ of $V(\Lambda)^*$.
With this setup we will introduce some special notation for the matrix coefficients, namely
\[
u^i_j = c^{\Lambda}_{f^i, v_j}(X) = f^i(X \triangleright v_j).
\]
We omit the dependence on the representation $V(\Lambda)$ to lighten the notation.
We will also denote by $\lambda_i$ the weight corresponding to the basis vector $v_i$.

\begin{remark}
Usually the quantized coordinate ring $\mathbb{C}_q[G]$ is presented in terms of generators coming from one particular representation of $\mathfrak{g}$.
For example, the presentation of the algebra $\mathbb{C}_q[SL(N)]$ in \cite{klsc} is given in terms of the generators $u^i_j$ which correspond to the choice of the fundamental representation.
Our general presentation here follows \cite{quantum-flag}, for example.
\end{remark}

Later on we will need some explicit formulae for the action of $\Uqg$ on $\Cqg$.
Let us write $X \triangleright v_i = \sum_j \pi(X)^j_i v_j$ for the representation. Then we obtain the formulae
\begin{equation}
\label{eq:action-ugen}
X \triangleright u^i_j = \sum_k \pi(X)^k_j u^i_k, \quad
X \triangleright u^{i*}_j = \sum_k \pi(S(X))^j_k u^{i*}_k.
\end{equation}
In obtaining the second one we have used the fact that $\{v_i\}_i$ is an orthonormal basis.
Similarly for the right action we obtain the formulae
\begin{equation}
\label{eq:actionR-ugen}
u^i_j \triangleleft X = \sum_k \pi(X)^i_k u^k_j, \quad
u^{i*}_j \triangleleft X = \sum_k \pi(S(X))^k_i u^{k*}_j.
\end{equation}

\section{Hochschild homology, quantum groups and projections}
\label{sec:hochschild}

In this section we will give a brief introduction to Hochschild homology, with emphasis on the twisted setting.
We will then recall the results of  Brown and Zhang on the Hochschild homology of certain Hopf algebras.
Finally we will discuss a simple method to obtain twisted $2$-cycles, valid for any algebra which admits projections satisfying certain properties.

\subsection{Hochschild homology}

Hochschild homology is a homology theory for associative algebras, which we consider here to be over $\mathbb{C}$.
The main reference for this section is \cite{loday}.
Let $A$ be an associative algebra and $M$ be an $A$-bimodule.
Write $C_n (A, M) := M \otimes A^{\otimes n}$. The \emph{Hochschild boundary} is the linear map $\mathrm{b}: C_n (A, M) \to C_{n - 1} (A, M)$ given by
\[
\begin{split}
\mathrm{b} (m \otimes a_1 \otimes \cdots \otimes a_n) & := m a_1 \otimes \cdots \otimes a_n \\
& + \sum_{i = 1}^{n - 1} (-1)^i m \otimes a_1 \otimes \cdots \otimes a_i a_{i + 1} \otimes \cdots \otimes a_n \\
& + (-1)^n a_n m \otimes a_1 \otimes \cdots \otimes a_{n - 1}.
\end{split}
\]
It satisfies $\mathrm{b}^2 = 0$, hence we have corresponding homology groups denoted by $H_\bullet(A, M)$.
We will also use the notation $HH_\bullet(A) = H_\bullet(A, A)$.
It can also be defined in terms of derived functors as $H_n(A, M) = \mathrm{Tor}^{A^e}_n(A, M)$, where $A^e := A \otimes A^\mathrm{op}$.
There is a corresponding dual cohomology theory, whose groups are denoted by $H^n(A, M)$.

A natural choice of bimodules is given by $M = A$.
Similarly we can consider the \emph{twisted bimodules} $M = {}_\sigma A$, which will be our main interest. They are defined as follows: as a vector space $M = A$, but the bimodule structure is given by $a \cdot b \cdot c = \sigma(a) b c $, where $\sigma \in \mathrm{Aut}(A)$.
For these we will use the notation $HH^\sigma_\bullet(A) := H_\bullet(A, {}_\sigma A)$.
We will also use the notation $\mathrm{b}_\sigma$ for the Hochschild boundary in this situation.
Notice that we could as well introduce a twist for the right multiplication, but as bimodules this gives nothing new.

An important case we want to consider is when $A$ is the algebra of functions on some smooth space $X$.
It turns out that the Hochschild homology of $A$ is related to the differential forms defined on $X$. This is the \emph{Hochschild-Kostant-Rosenberg} theorem, a proof of which can be found in \cite[Theorem 3.4.4]{loday}. We state the theorem for algebras over $\mathbb{C}$ for simplicity.
For a commutative unital algebra $A$, we have the $A$-module of differential forms $\Omega^\bullet_A := \bigwedge^\bullet_A \Omega^1_A$ constructed from the module of Kähler differentials $\Omega^1_A$, see \cite[Section 1.1.9]{loday}.

\begin{theorem}[Hochschild-Kostant-Rosenberg]
Let $A$ be a commutative smooth algebra over $\mathbb{C}$.
Then there is an isomorphism of graded $\mathbb{C}$-algebras $\Omega^\bullet_A \cong HH_\bullet(A)$.
\end{theorem}

We will not give the definition of a smooth algebra, but just mention that the example to keep in mind is $A = \mathbb{C}[X]$ for a smooth affine variety $X$.
The algebra structure on $HH_\bullet(A)$ is given by the shuffle product, which strongly relies on commutativity of $A$.

This result motivates a possible definition of differential forms for non-commutative algebras. However, as we will see below, in general $HH_\bullet (A)$ is very degenerate.

\subsection{The case of quantum groups}

The Hochschild homology of quantum $SU(2)$ and of the quantum $2$-sphere was computed by Masuda, Nakagami and Watanabe in the papers \cite{mnw-su2} and \cite{mnw-s2}.
Among their results we find that $HH_3(\mathbb{C}_q[SU(2)]) = 0$ and $HH_2(\mathbb{C}_q[S^2]) = 0$.
Therefore in this setting we do not have "volume forms".
The situation is different if we allow some twisting, namely by considering twisted bimodules as discussed above.
In this setting the computation for quantum $SU(2)$ was done by Hadfield and Krähmer in \cite{twisted-su2, sl2-second} and for the quantum $2$-sphere by Hadfield in \cite{twisted-s2}.

Motivated by these computations, Brown and Zhang made a general analysis of this phenomenon in \cite{brzh}.
The object of their study is the twisted Hochschild homology of a certain class of Hopf algebras, which includes the quantized coordinate rings $\Cqg$.
As a twist they consider a particular automorphism $\nu$, which generalized the classical \emph{Nakayama automorphism} for Frobenius algebras, which is unique up to inner automorphisms.
One of their main results is the following \cite[Theorem 3.4 and 5.3]{brzh}.

\begin{theorem}[Brown, Zhang]
Let $A$ be a Noetherian AS-Gorenstein Hopf algebra of finite global dimension $d$, with bijective antipode. Let $\nu$ be its Nakayama automorphism. Then we have $H_d(A, {}_{\nu^{-1}} A) \neq 0$ and $H^d(A, {}_\nu A) \neq 0$.
\end{theorem}

Moreover there is a twisted Poincaré duality connecting homology and cohomology.
This is a particularly simple instance of the general \emph{Van den Bergh duality} \cite{vandenbergh}.

\begin{theorem}[Brown, Zhang]
Let $A$ be as above. Then for any $A$-bimodule $M$ and for all $i$ we have $H^i(A, M) \cong H_{d - i}(A, {}_{\nu^{-1}} M)$.
\end{theorem}

These results can be applied to the quantized coordinate rings $\Cqg$.
In this case it is known that the finite global dimension $d$ coincides with the classical dimension.
Brown and Zhang show that $\nu$ is given by the modular automorphism in the case of $SL(N)$.
This is true in general by a result of Dolgushev \cite{dolgushev}, which uses techniques of deformation quantization.

\subsection{Twisted 2-cycles}

The aim of this paper is to study twisted $2$-cycles on the quantized coordinate rings $\Cqg$.
Below we discuss two reasons why this should be interesting.

1) The first reason is that the results of Brown and Zhang do not give concrete information about what happens in the intermediate degrees.
The bottom degree part $H_0(A, {}_{\nu^{-1}} A)$ can be determined explicitly from its definition, while the top degree part $H_d(A, {}_{\nu^{-1}} A)$ can be obtained using the twisted Poincaré duality mentioned above as
\[
H_d(A, {}_{\nu^{-1}} A) \cong H^0(A, A) \cong Z(A).
\]
Note that $Z(A) \neq 0$, since the center always contains the unit.
On the other hand we do not know the groups in the intermediate degrees. For example they could be all zero, which would be unsatisfactory for their interpretation as differential forms.

2) The second reason, which singles out $2$-cycles, is the following. At some point during our analysis we will naturally encounter quantum full flag manifolds corresponding to $\Cqg$. Classically full flag manifolds are \emph{Kähler manifolds}, a fact which more generally is true for any generalized flag manifold.
These admit a $2$-form $\omega$, called the \emph{Kähler form}, which among other things allows to obtain a volume form as $\omega^{\wedge n}$, where $n$ is the complex dimension.
Hence among twisted $2$-cycles we expect to find examples of quantum Kähler forms.
Differently from the commutative case, for non-commutative algebras there is no obvious way of multiplying classes.
But, if such a way exists after all, a natural question is whether one can obtain a top degree form by appropriately multiplying these quantum Kähler forms.

After this discussion, we will present a simple way to obtain twisted Hochschild $2$-cycles from projections satisfying suitable conditions.
A similar construction is used in \cite[Proposition 5.3]{wag09}.
Below $A$ will denote a general unital associative algebra.
We will make use of the trace map $\mathrm{Tr}: \mathrm{Mat}_r(A)^{\otimes n + 1} \to A^{\otimes n + 1}$, see \cite[Definition 1.2.1]{loday}. It is defined by
\[
\mathrm{Tr}(M_0 \otimes M_1 \otimes \cdots \otimes M_n) := \sum_{i_0, \cdots, i_n} (M_0)^{i_0}_{i_1} \otimes (M_1)^{i_1}_{i_2} \otimes \cdots \otimes (M_n)^{i_n}_{i_0}.
\]

\begin{lemma}
\label{lem:hoch-lemma}
Let $P \in \mathrm{Mat}(A)$ be a projection and $\sigma$ an automorphism of $A$.
Suppose there exists an invertible matrix $V \in \mathrm{Mat}(\mathbb{C})$ such that
\[
\sigma(P) = V P V^{-1}, \quad\mathrm{Tr}(V P) = c \cdot 1,
\]
for some $c \in \mathbb{C}$.
Consider the element $C(P) \in A^{\otimes 3}$ given by
\[
C(P) = \mathrm{Tr} \left( V (2 P - \mathrm{Id}) \otimes P \otimes P \right).
\]
Then we have a corresponding class $[C(P)] \in HH_2^\sigma(A)$.
\end{lemma}

\begin{proof}
Since we are in low dimension we can proceed with a direct computation.
Using the definition of the boundary map and of the $2$-chain $C(P)$ we obtain
\[
\begin{split}
\mathrm{b}_\sigma C(P)
& = \sum_{i, j, k, \ell} V^i_j (2 P^j_k - \delta^j_k) P^k_\ell \otimes P^\ell_i - \sum_{i, j, k, \ell} V^i_j (2 P^j_k - \delta^j_k) \otimes P^k_\ell P^\ell_i \\
& + \sum_{i, j, k, \ell} \sigma(P^\ell_i) V^i_j (2 P^j_k - \delta^j_k) \otimes P^k_\ell.
\end{split}
\]
Let us write $A_{1}$ and $A_{2}$ for the first and second line of
this expression. Using the projection relations $\sum_k P^i_k P^k_j = P^i_j$ and simplifying we get
\[
A_1 = - \sum_{i, j, \ell} V^i_j P^j_\ell \otimes P^\ell_i + 1 \otimes \sum_{i, j} V^i_j P^j_i.
\]
Since $V$ is assumed to be invertible, the second term can be rewritten as
\[
A_2 = \sum_{i, j, k, \ell, m, n} V^\ell_m (V^{-1})^m_n \sigma(P^n_i) V^i_j (2 P^j_k - \delta^j_k) \otimes P^k_\ell.
\]
Moreover using the condition $V^{-1}\sigma(P)V=P$ we find
\[
A_2 = \sum_{j, k, \ell, m} V^\ell_m P^m_j (2 P^j_k - \delta^j_k) \otimes P^k_\ell = \sum_{k, \ell, m} V^\ell_m P^m_k \otimes P^k_\ell.
\]
Finally summing the two terms we have a cancellation and we obtain
\[
\mathrm{b}_\sigma C(P) = A_1 + A_2 = 1 \otimes\sum_{i, j} V^i_j P^j_i.
\]
Now recall that the normalized Hochschild complex is defined in terms of the chains $\bar{C}_n (A) = A \otimes (A / \mathbb{C})^{\otimes n}$.
Hence using the condition $\mathrm{Tr}(V P) = c$ we conclude that $\mathrm{b}_\sigma C(P) = 0$ in the normalized Hochschild complex. Since this complex is quasi-isomorphic
to the usual Hochschild complex \cite[Proposition 1.1.15]{loday}, we obtain a class $[C(P)] \in HH^\sigma_2(A)$.
\end{proof}

\begin{remark}
The expression defining $C(P)$ can be seen as a modification of the Chern character $\mathrm{ch}_n : K_0(A) \mapsto H^\lambda_{2n}(A)$ given by $P \mapsto \mathrm{Tr} (P^{\otimes 2n + 1})$.
However such a simple modification, landing in Hochschild homology, seems to be possible only in the case $n = 1$.
\end{remark}

We will use this method to produce non-trivial classes for quantum full flag manifolds.

\section{Projections on quantized coordinate rings}
\label{sec:projections}

In this section we will define some projections on the quantized coordinate rings $\mathbb{C}_q[G]$. These will be built using some appropriate "matrix units", corresponding to the choice of an irreducible representation $V(\Lambda)$.
We will consider the action of the modular automorphism coming from the Haar state.
We will show that this action on the projections can be implemented by conjugation, provided a certain condition holds.

\subsection{Matrix units}

For the rest of this section we fix a representation $V(\Lambda)$ and denote by $u^i_j$ its
matrix coefficients with respect to an orthonormal basis, as explained before.

\begin{lemma}
The matrix coefficients $u^i_j$ satisfy the relations
\[
\begin{gathered}\sum_{k}u_{a}^{k*}u_{b}^{k}=\delta_{b}^{a}1=\sum_{k}u_{k}^{a}u_{k}^{b*},\\
\sum_{k}q^{(2\rho,\lambda_{k}-\lambda_{b})}u_{b}^{k}u_{a}^{k*}=\delta_{b}^{a}1=\sum_{k}q^{(2\rho,\lambda_{a}-\lambda_{k})}u_{k}^{b*}u_{k}^{a}.
\end{gathered}
\]
\end{lemma}
\begin{proof}
Recall that in a Hopf algebra we have $S(a_{(1)}) a_{(2)} = \varepsilon(a) 1 = a_{(1)} S(a_{(2)})$
for all $a$. We apply this identity to $u^a_b$. We have
$\Delta(u^a_b) = \sum_k u^a_k \otimes u^k_b$ and $\varepsilon(u_{b}^{a})=\delta_{b}^{a}$.
Then
\[
\sum_{k}S(u_{k}^{a})u_{b}^{k}=\delta_{b}^{a}1=\sum_{k}u_{k}^{a}S(u_{b}^{k}).
\]
Using $S(u_{j}^{i})=u_{i}^{j*}$ it can be rewritten as claimed.

Next we apply the above identity to $S(u^a_b)$. For the counit and the coproduct we have $\varepsilon(S(u^a_b)) = \delta^a_b$
and $\Delta(S(u^a_b)) = \sum_k S(u^k_b) \otimes S(u^a_k)$.
Then we obtain
\[
\sum_k S^2(u^k_b) S(u^a_k) = \delta^a_b 1 = \sum_k S(u^k_b) S^2(u^a_k).
\]
We need to use the identity $S^2(u^i_j) = q^{(2 \rho, \lambda_i - \lambda_j)} u^i_j$.
Plugging this in we find
\[
\sum_{k}q^{(2\rho,\lambda_{k}-\lambda_{b})}u_{b}^{k}S(u_{k}^{a})=\delta_{b}^{a}1=\sum_{k}q^{(2\rho,\lambda_{a}-\lambda_{k})}S(u_{b}^{k})u_{k}^{a}.
\]
Using $S(u_{j}^{i})=u_{i}^{j*}$ it can be rewritten as claimed.
\end{proof}

\begin{remark}
We could avoid working with orthonormal bases and express everything in terms of $S(u^j_i) = u^{i*}_j$, but this makes many of the following formulae less clear.
\end{remark}

We will now define some "matrix units" in terms of the elements $u^i_j$ and $u^{i*}_j$.
For any $m, n$ corresponding to the basis of $V(\Lambda)$, we define the matrices $\mathsf{M}_{m}^{n}, \mathsf{N}_{m}^{n} \in \mathrm{Mat}(\CqG)$ by
\[
(\mathsf{M}^n_m)^i_j := u^{m*}_i u^n_j, \quad
(\mathsf{N}^n_m)^i_j := u^i_m u^{j*}_n.
\]

\begin{proposition}
\label{prop:mat-unit1}
1) The matrices $\{ \mathsf{M}_{m}^{n} \}_{m, n}$ are linearly independent and satisfy
\[
(\mathsf{M}_{m}^{n})^{*}=\mathsf{M}_{n}^{m},\quad\mathsf{M}_{m}^{n}\mathsf{M}_{o}^{p}=\delta_{o}^{n}\mathsf{M}_{m}^{p},\quad\mathrm{Tr}(\pi(K_{2\rho}^{-1})\mathsf{M}_{m}^{n})=\delta_{m}^{n}q^{-(2\rho,\lambda_{m})}.
\]
2) The matrices $\{ \mathsf{N}_{m}^{n} \}_{m, n}$ are linearly independent and satisfy
\[
(\mathsf{N}_{m}^{n})^{*}=\mathsf{N}_{n}^{m},\quad\mathsf{N}_{m}^{n}\mathsf{N}_{o}^{p}=\delta_{o}^{n}\mathsf{N}_{m}^{p},\quad\mathrm{Tr}(\pi(K_{2\rho})\mathsf{N}_{m}^{n})=\delta_{m}^{n}q^{(2\rho,\lambda_{m})}.
\]
\end{proposition}

\begin{proof}
1) First we prove linear independence. Suppose $\sum_{m,n}c_{n}^{m}\mathsf{M}_{m}^{n}=0$.
Taking the counit of the $(i,j)$-component we get $\sum_{m,n}c_{n}^{m}\varepsilon(\mathsf{M}_{m}^{n})_{j}^{i}=c_{j}^{i}$,
where we have used $\varepsilon(\mathsf{M}_{m}^{n})_{j}^{i}=\delta_{m}^{i}\delta_{j}^{n}$.
This shows that $c_{j}^{i}=0$ for all $i$ and $j$, that is the
matrices $\mathsf{M}_{m}^{n}$ are linearly independent. Next it is
immediate that $(\mathsf{M}_{m}^{n})_{j}^{i*}=u_{j}^{n*}u_{i}^{m}=(\mathsf{M}_{n}^{m})_{i}^{j}$.
For the product relation we compute
\[
\sum_{k}(\mathsf{M}_{m}^{n})_{k}^{i}(\mathsf{M}_{o}^{p})_{j}^{k}=u_{i}^{m*}\left(\sum_{k}u_{k}^{n}u_{k}^{o*}\right)u_{j}^{p}=\delta_{o}^{n}u_{i}^{m*}u_{j}^{p}=\delta_{o}^{n}(\mathsf{M}_{m}^{p})_{j}^{i}.
\]
Finally for the $q$-trace relation we have
\[
\sum_{i}q^{-(2\rho,\lambda_{i})}(\mathsf{M}_{m}^{n})_{i}^{i}=q^{-(2\rho,\lambda_{n})}\sum_{i}q^{(2\rho,\lambda_{n}-\lambda_{i})}u_{i}^{m*}u_{i}^{n}=\delta_{m}^{n}q^{-(2\rho,\lambda_{m})}.
\]

2) Linear independence is proven in the case of $\{ \mathsf{M}_{m}^{n} \}_{m, n}$. Similarly we have $(\mathsf{N}_{m}^{n})_{j}^{i*}=(\mathsf{N}_{n}^{m})_{i}^{j}$. For
the product relation we compute
\[
\sum_{k}(\mathsf{N}_{m}^{n})_{k}^{i}(\mathsf{N}_{o}^{p})_{j}^{k}=u_{m}^{i}\left(\sum_{k}u_{n}^{k*}u_{o}^{k}\right)u_{p}^{j*}=\delta_{o}^{n}u_{m}^{i}u_{p}^{j*}=\delta_{o}^{n}(\mathsf{N}_{m}^{p})_{j}^{i}.
\]
Finally for the $q$-trace relation we have
\[
\sum_{i}q^{(2\rho,\lambda_{i})}(\mathsf{N}_{m}^{n})_{i}^{i}=q^{(2\rho,\lambda_{m})}\sum_{i}q^{(2\rho,\lambda_{i}-\lambda_{m})}u_{m}^{i}u_{n}^{i*}=\delta_{m}^{n}q^{(2\rho,\lambda_{m})}.
\qedhere
\]
\end{proof}
\begin{remark}
These relations are essentially those of the matrix units $m_{m}^{n}$
which are $1$ in the $(m,n)$-entry and zero elsewhere, that is $(m_{m}^{n})_{j}^{i}=\delta_{m}^{i}\delta_{n}^{j}$
(with respect to an orthonormal basis).
\end{remark}

We can build more general matrices in terms of these matrix units. In particular within this setting it is easy to state when such matrices correspond to projections.

\begin{lemma}
Let $\mathsf{P} = \sum_{m, n} c^m_n \mathsf{M}^n_m$ and $\mathsf{Q} = \sum_{m, n} c^m_n \mathsf{N}^n_m$.
We have that $\mathsf{P}$ and $\mathsf{Q}$ are projections if and
only if $\sum_\ell c^m_\ell c^\ell_n = c^m_n$. They are orthogonal projections if moreover $\overline{c^m_n} = c^n_m$.
\end{lemma}

\begin{proof}
For the relation $\mathsf{P}^2 = \mathsf{P}$ we use the product rule for $\mathsf{M}^n_m$ and compute
\[
\sum_k \mathsf{P}_{k}^{i} \mathsf{P}_{j}^{k} = \sum_{m,n,o,p} c_{n}^{m} c_{p}^{o} \sum_{k} (\mathsf{M}_{m}^{n})_{k}^{i} (\mathsf{M}_{o}^{p})_{j}^{k} = \sum_{m,n,p} c_{n}^{m} c_{p}^{n} (\mathsf{M}_{m}^{p})_{j}^{i}.
\]
Since the matrix units $\mathsf{M}^n_m$ are linearly independent, we obtain $\sum_{k} \mathsf{P}_{k}^{i} \mathsf{P}_{j}^{k} = \mathsf{P}_{j}^{i}$ if and only if $\sum_{n} c_{n}^{m} c_{p}^{n} = c_{p}^{m}$.
For the orthogonality condition we compute
\[
(\mathsf{P})_{j}^{i*}=\sum_{m,n}\overline{c_{n}^{m}}(\mathsf{M}_{m}^{n})_{j}^{i*}=\sum_{m,n}\overline{c_{n}^{m}}(\mathsf{M}_{n}^{m})_{i}^{j}.
\]
Hence $(\mathsf{P})_{j}^{i*} = (\mathsf{P})_{i}^{j}$ if and only if
$\overline{c_{n}^{m}} = c_{m}^{n}$.
Finally we observe that we get the same results for $\mathsf{Q}$, since the matrix units $\mathsf{N}_{m}^{n}$ have the same product rule and action of $*$ as $\mathsf{M}_{m}^{n}$.
\end{proof}

We will use the notations $\mathsf{P} = \sum_{m, n} c^m_n \mathsf{M}^n_m$ and $\mathsf{Q} = \sum_{m, n} c^m_n \mathsf{N}^n_m$ throughout the paper.

\subsection{Modular element}

A natural twist to consider is the \emph{modular automorphism} $\theta : \mathbb{C}_q[G] \to \mathbb{C}_q[G]$ (or its inverse) coming from the Haar state. It is given explicitly by
\[
\theta(a) = K_{2 \rho} \triangleright a \triangleleft K_{2 \rho}.
\]
It satisfies the following property: if we denote by $h : \mathbb{C}_q[G] \to \mathbb{C}$ the Haar state, then we have $h(a b) = h(\theta(b) a)$ for all $a, b \in  \mathbb{C}_q[G]$.
It is useful to consider a more general situation.

\begin{notation}
\label{not:sigma-lambda}
Given two weights $\lambda, \lambda^\prime$ we write $\sigma_{\lambda, \lambda^\prime}(a) := K_\lambda \triangleright a \triangleleft K_{\lambda^\prime}$.
\end{notation}

Therefore $\sigma_{\lambda, \lambda^\prime}$ expresses a general action coming from the Cartan generators.
In the next lemma we compute this action on the entries of the matrices $\mathsf{M}^n_m$ and $\mathsf{N}^n_m$.

\begin{lemma}
\label{lem:action-k}
We have the formulae
\[
\begin{split}
\sigma_{\lambda, \lambda^\prime} (\mathsf{M}^n_m)^i_j
& = q^{-(\lambda, \lambda_i - \lambda_j)}  q^{-(\lambda^\prime, \lambda_m - \lambda_n)} (\mathsf{M}^n_m)^i_j, \\
\sigma_{\lambda, \lambda^\prime} (\mathsf{N}^n_m)^i_j
& = q^{(\lambda, \lambda_m - \lambda_n)} q^{(\lambda^\prime, \lambda_i - \lambda_j)} (\mathsf{N}^n_m)^i_j.
\end{split}
\]
\end{lemma}

\begin{proof}
We immediately compute $K_\lambda \triangleright u^a_b \triangleleft K_{\lambda^\prime} = q^{(\lambda, \lambda_b)} q^{(\lambda^\prime, \lambda_a)} u^a_b$.
Next recall the identities
\[
X\triangleright a^{*}=(S(X)^{*}\triangleright a)^{*},\quad a^{*}\triangleleft X=(a\triangleleft S(X)^{*})^{*}.
\]
Since $S(K_{\lambda})^{*}=K_{\lambda}^{-1}$ we have
\[
K_{\lambda}\triangleright u_{b}^{a*}\triangleleft K_{\lambda^{\prime}}=(K_{\lambda}^{-1}\triangleright u_{b}^{a}\triangleleft K_{\lambda^{\prime}}^{-1})^{*}=q^{-(\lambda,\lambda_{b})}q^{-(\lambda^{\prime},\lambda_{a})}u_{b}^{a*}.
\]
Therefore for $(\mathsf{M}^n_m)^i_j = u_{i}^{m*}u_{j}^{n}$
we have
\[
K_\lambda \triangleright (\mathsf{M}^n_m)^i_j \triangleleft K_{\lambda^{\prime}} = q^{-(\lambda, \lambda_i - \lambda_j)}  q^{-(\lambda^\prime, \lambda_m - \lambda_n)} (\mathsf{M}^n_m)^i_j.
\]
Similarly for $(\mathsf{N}^n_m)^i_j = u_{m}^{i}u_{n}^{j*}$
we have
\[
K_{\lambda} \triangleright (\mathsf{N}^n_m)^i_j \triangleleft K_{\lambda^{\prime}}=q^{(\lambda,\lambda_{m}-\lambda_{n})}q^{(\lambda^{\prime},\lambda_{i}-\lambda_{j})} (\mathsf{N}^n_m)^i_j.
\qedhere
\]
\end{proof}

We now investigate the case of the modular automorphism $\theta$.

\begin{proposition}
\label{prop:conjugation1}
Suppose that $c^m_n = 0$ if $\lambda_m \neq \lambda_n$. Then we have the relations
\[
\theta(\mathsf{P})=\pi(K_{2\rho}^{-1})\mathsf{P}\pi(K_{2\rho}),\quad\theta(\mathsf{Q})=\pi(K_{2\rho})\mathsf{Q}\pi(K_{2\rho}^{-1}),
\]
where the automorphism $\theta$ is applied entrywise.
\end{proposition}

\begin{proof}
Using the formulae in \cref{lem:action-k} we compute
\[
\begin{split}
\theta(\mathsf{M}^n_m)^i_j
& = q^{-(2\rho, \lambda_i - \lambda_j)} q^{-(2\rho, \lambda_m - \lambda_n)} (\mathsf{M}^n_m)^i_j \\
& = q^{-(2\rho, \lambda_m - \lambda_n)} \pi(K_{2\rho}^{-1})_{i}^{i} (\mathsf{M}^n_m)^i_j \pi(K_{2\rho})_{j}^{j}.
\end{split}
\]
Therefore for the matrix $\mathsf{P}$ we obtain
\[
\theta(\mathsf{P}_{j}^{i}) = \sum_{m, n} c^m_n q^{-(2\rho, \lambda_m - \lambda_n)} \pi(K_{2\rho}^{-1})_{i}^{i}(\mathsf{M}_{m}^{n})_{j}^{i}\pi(K_{2\rho})_{j}^{j}.
\]
Under the assumption on the coefficients $c^m_n$ we have the identity $c^m_n q^{-(2\rho, \lambda_m - \lambda_n)} = c^m_n$, hence we obtain the result.
Similarly, for the matrix $\mathsf{Q}$ we compute
\[
\begin{split}
\theta(\mathsf{N}_{m}^{n})_{j}^{i} & =q^{(2\rho,\lambda_{m}-\lambda_{n})}q^{(2\rho,\lambda_{i}-\lambda_{j})}(\mathsf{N}_{m}^{n})_{j}^{i}\\
 & =q^{(2\rho,\lambda_{m}-\lambda_{n})}\pi(K_{2\rho})_{i}^{i}(\mathsf{N}_{m}^{n})_{j}^{i}\pi(K_{2\rho}^{-1})_{j}^{j}.
\end{split}
\]
Then the conclusion is immediate.
\end{proof}

\begin{remark}
The condition on the coefficients $c_{n}^{m}$ is clearly not necessary
for $\mathsf{P}$ to be an eigenvector, as can be seen by considering
$\mathsf{P} = \mathsf{M}_{m}^{n}$ with $\lambda_m \neq \lambda_n$. It is also easy to see that not
all $\mathsf{P}$ are eigenvectors. For example consider $\mathsf{P}=\mathsf{M}_{m}^{n}+\mathsf{M}_{n}^{m}$.
Then
\[
\theta(\mathsf{P}_{j}^{i})=q^{-(2\rho,\lambda_{i}-\lambda_{j})}(q^{(2\rho,\lambda_{n}-\lambda_{m})}(\mathsf{M}_{m}^{n})_{j}^{i}+q^{-(2\rho,\lambda_{n}-\lambda_{m})}(\mathsf{M}_{n}^{m})_{j}^{i}).
\]
This is an eigenvector if and only if $(2\rho,\lambda_{n}-\lambda_{m})=0$.
\end{remark}

\section{Quantum flag manifolds and equivariant K-theory}
\label{sec:flag}

In this section we will connect the results obtained in the previous section with quantum flag manifolds and equivariant $K$-theory.
First we show that the condition we assumed for the coefficients $c^m_n$ is precisely the condition for the matrices $\mathsf{P}$ and $\mathsf{Q}$ to descend to the appropriate quantum full flag manifolds.
Secondly, we show that the projections built from the matrix units $\mathsf{M}^n_m$ and $\mathsf{N}^n_m$ belong to appropriate equivariant $K$-theory groups.

\subsection{Connection with full flag manifolds}

Classically the full flag manifold $G/T$ is defined as the quotient of $G$ by the maximal torus $T$.
Functions on these manifolds are then functions on $G$ which are invariant under the action of $T$. Equivalently these are functions which are invariant under the action of the Cartan subalgebra.
In the quantum setting the role of the Cartan generators is played by the generators $K_\lambda$.
This discussion naturally leads to define (functions on) quantum full flag manifolds as follows
\[
\mathbb{C}_q[G/T] := \{a \in \mathbb{C}_q[G] : K_\lambda \triangleright a = a\}, \quad
\mathbb{C}_q[T\backslash G] := \{a \in \mathbb{C}_q[G] : a \triangleleft K_\lambda = a\}.
\]
As a reference for these quantum homogeneous spaces see \cite{quantum-flag}, for example.
Recall that we have commuting left and right actions of $\Uqg$ on the quantized coordinate ring $\Cqg$. Hence we get a right action of $\Uqg$ on $\mathbb{C}_q[G/T]$ and a left action of $\Uqg$ on $\mathbb{C}_q[T\backslash G]$.

We will now show that the condition on the coefficients $c^m_n$ appearing in \cref{prop:conjugation1} can be interpreted geometrically as follows: it is precisely the condition for the matrices $\mathsf{P}$ and $\mathsf{Q}$ to descend to the appropriate quantum full flag manifolds.

\begin{proposition}
\label{prop:proj-flag}
Let $\mathsf{P} = \sum_{m, n} c^m_n \mathsf{M}^n_m$ and $\mathsf{Q} = \sum_{m, n} c^m_n \mathsf{N}^n_m$.
Then:

1) we have $\mathsf{P} \in \mathrm{Mat}(\flagR)$ if and only if $c^m_n = 0$ for $\lambda_m \neq \lambda_n$,

2) we have $\mathsf{Q} \in \mathrm{Mat}(\flagL)$ if and only if $c^m_n = 0$ for $\lambda_m \neq \lambda_n$.
\end{proposition}

\begin{proof}
1) We have to check when all the entries $\mathsf{P}^i_j$ belong to $\flagR$. Recall that from \cref{lem:action-k} we have $(\mathsf{M}^n_m)^i_j \triangleleft  K_\lambda = q^{-(\lambda, \lambda_m - \lambda_n)} (\mathsf{M}^n_m)^i_j$.
Then we compute
\[
\mathsf{P}^i_j \triangleleft  K_\lambda = \sum_{m, n} c^m_n (\mathsf{M}^n_m)^i_j \triangleleft  K_\lambda = \sum_{m, n} q^{-(\lambda, \lambda_m - \lambda_n)} c^m_n (\mathsf{M}^n_m)^i_j.
\]
Now consider the condition $\mathsf{P}^i_j \triangleleft  K_\lambda = \mathsf{P}^i_j$.
Since the matrices $\mathsf{M}^n_m$ are linearly independent we must have $q^{-(\lambda, \lambda_m - \lambda_n)} c^m_n = c^m_n$ for all $m$ and $n$. But $(\lambda, \lambda_m - \lambda_n) = 0$ for all $\lambda$ holds if and only if $\lambda_m = \lambda_n$, by non-degeneracy. Hence we must have $c^m_n = 0$ for $\lambda_m \neq \lambda_n$.

2) The proof for $\mathsf{Q}$ is completely analogous and we omit it.
\end{proof}

\begin{remark}
It is clear from the result above that $\mathsf{M}^m_m \in \mathrm{Mat}(\flagR)$ and $\mathsf{N}^m_m \in \mathrm{Mat}(\flagL)$.
These are $N \times N$ matrices, where $N$ is the dimension of the fixed representation $V(\Lambda)$.
\end{remark}

\subsection{Equivariant K-theory}

In this subsection we show that the projections built using the matrix units $\mathsf{M}^n_m$ and $\mathsf{N}^n_m$ belong to appropriate equivariant $K$-theory groups.
The setting we consider is based on \cite{netu04} (see also the references therein for the general case of coactions of locally compact quantum groups), but we follow the presentation given in \cite[Section 3]{wag09} (without taking into account the $*$-structure, for simplicity).

Let $\mathcal{U}$ be a Hopf algebra and $\mathcal{B}$ be a right $\mathcal{U}$-module algebra. Let $\rho^\circ: \mathcal{U}^\circ \to \mathrm{End}(\mathbb{C}^N)$ be a representation of the opposite algebra $\mathcal{U}^\circ$, or equivalently we take $\rho^\circ$ to be an anti-homomorphism.
We have an embedding of $\mathrm{Mat}_{N \times N}(\mathbb{C}) \otimes \mathcal{B}$ into $\mathrm{End}(\mathbb{C}^N \otimes \mathcal{B})$ given by $T \otimes b \mapsto T \otimes L_b$, where $L_b$ denotes left multiplication by $b \in \mathcal{B}$.
Working in this setup, we can write the action of a matrix in $\mathrm{Mat}_{N \times N}(\mathcal{B}) \cong \mathrm{Mat}_{N \times N}(\mathbb{C}) \otimes \mathcal{B}$ on a column vector in $\mathcal{B}^N \cong \mathbb{C}^N \otimes \mathcal{B}$ in terms of the usual rules of matrix multiplication.

The algebra $\mathrm{End}(V \otimes \mathcal{B})$ becomes a left $\mathcal{U}^\circ$-module with respect to the left adjoint action of $\mathcal{U}^\circ$. It can be shown that, with respect to this action, the algebra $\mathrm{Mat}_{N \times N}(\mathcal{\mathcal{B}})$ becomes a left $\mathcal{U}^\circ$-module subalgebra of $\mathrm{End}(V \otimes \mathcal{B})$. The explicit action $\mathrm{ad}^\circ_L$ of $\mathcal{U}^\circ$ is given by
\[
\mathrm{ad}^\circ_L(X)(M) = \rho^\circ(X_{(1)}) (M \triangleleft X_{(2)}) \rho^\circ(S^{-1}(X_{(3)})),
\quad X \in \mathcal{U}, \ M \in \mathrm{Mat}_{N \times N}(\mathcal{B}).
\]
Here $M \triangleleft X$ means the action of $X$ on each entry of the matrix $M$. Note that we can consider equivalently $\mathrm{Mat}_{N \times N}(\mathcal{B})$ as a right $\mathcal{U}$-module subalgebra.

A matrix $M \in \mathrm{Mat}_{N \times N}(\mathcal{B})$ is called \emph{right $\mathcal{U}$-invariant} if there exists a representation $\rho: \mathcal{U}^\circ \to \mathrm{End}(\mathbb{C}^N)$ such that $\mathrm{ad}^\circ_L (X)(M) = \varepsilon(X) M$ for all $X \in \mathcal{U}$.
We can introduce a notion of (Murray-von Neumann) equivalence on invariant projections, see \cite[Definition 3.1]{wag09}.
The Grothendieck group of equivalence classes of invariant projections is denoted by ${}^\mathcal{U} K_0(\mathcal{B})$, which we call the \emph{$\mathcal{U}$-equivariant $K_0$-group of $\mathcal{B}$}.

The situation is analogous if we consider $\mathcal{B}$ to be a left $\mathcal{U}$-module algebra. In this case the algebra $\mathrm{Mat}_{N \times N}(\mathcal{B})$ becomes a right $\mathcal{U}^\circ$-module subalgebra and the action is given by
\[
\mathrm{ad}^\circ_R(X)(M) = \rho^\circ(S^{-1}(X_{(1)})) (S^{-2}(X_{(2)}) \triangleright M) \rho^\circ(X_{(3)}),
\quad X \in \mathcal{U}, \ M \in \mathrm{Mat}_{N \times N}(\mathcal{B}).
\]
Equivalently $\mathrm{Mat}_{N \times N}(\mathcal{\mathcal{B}})$ is a left $\mathcal{U}$-module subalgebra.
The condition for a matrix $M \in \mathrm{Mat}_{N \times N}(\mathcal{B})$ to be \emph{left $\mathcal{U}$-invariant} is then $\mathrm{ad}^\circ_R (X)(M) = \varepsilon(X) M$ for all $X \in \mathcal{U}$.
The corresponding $\mathcal{U}$-equivariant $K_0$-group is denoted by $K_0(\mathcal{B})^\mathcal{U}$.

We are interested in the situation where $\mathcal{U} = \Uqg$ and $\mathcal{B} = \Cqg$, which is naturally a $\Uqg$-bimodule algebra. Taking an $N$-dimensional representation $V(\Lambda)$ of $\Uqg$, we obtain elements $\mathsf{M}^n_m, \mathsf{N}^n_m \in \mathrm{Mat}_{N \times N} (\Cqg)$ by \cref{prop:mat-unit1}.

\begin{lemma}
\label{lem:equiv-mat}
Let $X \in \Uqg$. Then we have
\[
X\triangleright\mathsf{M}_{m}^{n}=\pi(S(X_{(1)}))\mathsf{M}_{m}^{n}\pi(X_{(2)}),\quad\mathsf{N}_{m}^{n}\triangleleft X=\pi(X_{(1)})\mathsf{N}_{m}^{n}\pi(S(X_{(2)})).
\]
\end{lemma}

\begin{proof}
Using the formulae in \eqref{eq:action-ugen} we compute
\[
\begin{split}
X\triangleright(\mathsf{M}_{m}^{n})_{j}^{i} & =(X_{(1)}\triangleright u_{i}^{m*})(X_{(2)}\triangleright u_{j}^{n}
= \sum_{k,\ell}\pi(S(X_{(1)}))_{k}^{i}u_{k}^{m*}\pi(X_{(2)})_{j}^{\ell}u_{\ell}^{n}\\
 & =\sum_{k,\ell}\pi(S(X_{(1)}))_{k}^{i}(\mathsf{M}_{m}^{n})_{\ell}^{k}\pi(X_{(2)})_{j}^{\ell}.
\end{split}
\]
Similarly for the right action, using the formulae in \eqref{eq:actionR-ugen}, we compute
\[
\begin{split}(\mathsf{N}_{m}^{n})_{j}^{i}\triangleleft X & =(u_{m}^{i}\triangleleft X_{(1)})(u_{n}^{j*}\triangleleft X_{(2)})
= \sum_{k,\ell}\pi(X_{(1)})_{k}^{i}u_{m}^{k}\pi(S(X_{(2)}))_{j}^{\ell}u_{n}^{\ell*}\\
 & =\sum_{k,\ell}\pi(X_{(1)})_{k}^{i}(\mathsf{N}_{m}^{n})_{\ell}^{k}\pi(S(X_{(2)}))_{j}^{\ell}.
\end{split}
\]
Rewriting these identities in matrix notation gives the result.
\end{proof}

We can now easily show that these elements are invariant.

\begin{proposition}
1) Let $\rho^{\circ}: \Uqg \to \mathrm{End}(V)$ be the anti-homomorphism defined by
\[
\rho^{\circ}(X)=\pi(K_{2\rho}^{-1}S(X)K_{2\rho}).
\]
Then $\mathsf{M}^n_m$ is left $\Uqg$-invariant, that is $\mathrm{ad}_{R}^{\circ}(X)(\mathsf{M}_{m}^{n})=\varepsilon(X)\mathsf{M}_{m}^{n}$ for all $X \in \Uqg$.

2) Let $\rho^{\circ}: \Uqg \to \mathrm{End}(V)$ be the anti-homomorphism defined by
\[
\rho^{\circ}(X)=\pi(K_{2\rho}S^{-1}(X)K_{2\rho}^{-1}).
\]
Then $\mathsf{N}^n_m$ is right $\Uqg$-invariant, that is $\mathrm{ad}_{L}^{\circ}(X)(\mathsf{N}_{m}^{n})=\varepsilon(X)\mathsf{N}_{m}^{n}$ for all $X \in \Uqg$.
\end{proposition}

\begin{proof}
1) It is immediate to check that $\rho^{\circ}(X)=\pi(K_{2\rho}^{-1}S(X)K_{2\rho})$ is an anti-homomorphism. Plugging this expression into the definition of $\mathrm{ad}_{R}^{\circ}$ we get
\[
\mathrm{ad}_{R}^{\circ}(X)(\mathsf{M}_{m}^{n}) =\pi(K_{2\rho}^{-1}X_{(1)}K_{2\rho})(S^{-2}(X_{(2)})\triangleright\mathsf{M}_{m}^{n})\pi(K_{2\rho}^{-1}S(X_{(3)})K_{2\rho}).
\]
From \cref{lem:equiv-mat} it follows that $S^{-2}(X)\triangleright\mathsf{M}_{m}^{n}=\pi(S^{-1}(X_{(1)}))\mathsf{M}_{m}^{n}\pi(S^{-2}(X_{(2)}))$.
Then
\[
\mathrm{ad}_{R}^{\circ}(X)(\mathsf{M}_{m}^{n}) =\pi(K_{2\rho}^{-1}X_{(1)}K_{2\rho}S^{-1}(X_{(2)}))\mathsf{M}_{m}^{n}\pi(S^{-2}(X_{(3)})K_{2\rho}^{-1}S(X_{(4)})K_{2\rho}).
\]
Recall that $S^{2}(X)=K_{2\rho}XK_{2\rho}^{-1}$. From this we obtain the relations $K_{2\rho}S^{-1}(X)=S(X)K_{2\rho}$ and $S^{-2}(X)=K_{2\rho}^{-1}XK_{2\rho}$.
Plugging them in we get
\[
\begin{split}\mathrm{ad}_{R}^{\circ}(X)(\mathsf{M}_{m}^{n}) & =\pi(K_{2\rho}^{-1}X_{(1)}S(X_{(2)})K_{2\rho})\mathsf{M}_{m}^{n}\pi(K_{2\rho}^{-1}X_{(3)}S(X_{(4)})K_{2\rho})\\
 & =\pi(K_{2\rho}^{-1}\varepsilon(X_{(1)})K_{2\rho})\mathsf{M}_{m}^{n}\pi(K_{2\rho}^{-1}\varepsilon(X_{(2)})K_{2\rho})\\
 & =\varepsilon(X_{(1)})\mathsf{M}_{m}^{n}\varepsilon(X_{(2)})=\varepsilon(X)\mathsf{M}_{m}^{n}.
\end{split}
\]

2) Similarly to the previous case it is easy to check that $\rho^{\circ}(X)=\pi(K_{2\rho}S^{-1}(X)K_{2\rho}^{-1})$ is an anti-homomorphism. Plugging this expression into the definition of $\mathrm{ad}_{L}^{\circ}$ we get
\[
\mathrm{ad}_{L}^{\circ}(X)(\mathsf{N}_{m}^{n}) =\pi(K_{2\rho}S^{-1}(X_{(1)})K_{2\rho}^{-1})(\mathsf{N}_{m}^{n}\triangleleft X_{(2)})\pi(K_{2\rho}S^{-2}(X_{(3)})K_{2\rho}^{-1}).
\]
Using $\mathsf{N}_{m}^{n}\triangleleft X=\pi(X_{(1)})\mathsf{N}_{m}^{n}\pi(S(X_{(2)}))$ from \cref{lem:equiv-mat} we get
\[
\mathrm{ad}_{L}^{\circ}(X)(\mathsf{N}_{m}^{n}) =\pi(K_{2\rho}S^{-1}(X_{(1)})K_{2\rho}^{-1}X_{(2)})\mathsf{N}_{m}^{n}\pi(S(X_{(3)})K_{2\rho}S^{-2}(X_{(4)})K_{2\rho}^{-1}).
\]
We use the identities $S^{-1}(X)K_{2\rho}^{-1}=K_{2\rho}^{-1}S(X)$
and $S^{-2}(X)=K_{2\rho}^{-1}XK_{2\rho}$. Then
\[
\begin{split}\mathrm{ad}_{L}^{\circ}(X)(\mathsf{N}_{m}^{n}) & =\pi(S(X_{(1)})X_{(2)})\mathsf{N}_{m}^{n}\pi(S(X_{(3)})X_{(4)})\\
 & =\varepsilon(X_{(1)})\mathsf{N}_{m}^{n}\varepsilon(X_{(2)})=\varepsilon(X)\mathsf{N}_{m}^{n}.
\qedhere
\end{split}
\]
\end{proof}

\begin{corollary}
Let $\mathsf{P} = \sum_{m, n} c^m_n \mathsf{M}^n_m$ and $\mathsf{Q} = \sum_{m, n} c^m_n \mathsf{N}^n_m$. Suppose they are projections. Then $\mathsf{P} \in K_0(\Cqg)^{\Uqg}$ and $\mathsf{Q} \in {}^{\Uqg}K_0(\Cqg)$.
\end{corollary}

\begin{proof}
By the previous proposition $\mathsf{M}^n_m$ is left $\Uqg$-invariant and $\mathsf{N}^n_m$ is right $\Uqg$-invariant. Then the result follows immediately from the definitions.
\end{proof}

\section{Twisted 2-cycles and 2-cocycles}
\label{sec:twisted}

In this section we will show, using the results of the previous sections, that we obtain classes in the twisted Hochschild homology of $\Cqg$. Moreover these naturally descend to appropriate quantum full flag manifolds.
In order to prove their non-triviality, we introduce some appropriate twisted $2$-cocycles.
The pairings will be computed in the next section.

\subsection{Twisted 2-cycles}

First we deal with the twisted homology classes. Here the natural twist to consider is given by $\theta$, the modular automorphism of the Haar state.

\begin{proposition}
\label{prop:twisted-cycles}
Let $\mathsf{P}, \mathsf{Q}$ be projections and suppose that $c^m_n = 0$ if $\lambda_m \neq \lambda_n$. Define
\[
\begin{split}
C(\mathsf{P}) & := \mathrm{Tr}\left(\pi(K_{2\rho}^{-1})(2\mathsf{P}-\mathrm{Id})\otimes\mathsf{P}\otimes\mathsf{P}\right),\\
C(\mathsf{Q}) & := \mathrm{Tr}\left(\pi(K_{2\rho})(2\mathsf{Q}-\mathrm{Id})\otimes\mathsf{Q}\otimes\mathsf{Q}\right).
\end{split}
\]
Then we obtain classes $[C(\mathsf{P})], \ [C(\mathsf{Q})] \in HH_2^\theta (\mathbb{C}_q[G])$.
\end{proposition}

\begin{proof}
To prove this result we will use \cref{lem:hoch-lemma}. Recall that this states that, given a projection $P \in \mathrm{Mat}(A)$, the $2$-chain $C(P) = \mathrm{Tr}(V(2 P - \mathrm{Id}) \otimes P \otimes P) \in A^{\otimes 3}$ defines a class in $HH^\sigma_2(A)$ if there exists an invertible matrix $V$ such that
\[
\mathrm{Tr} (V P) = c \cdot 1, \quad
\sigma(P) = V P V^{-1}.
\]
The first condition is satisfied, since from \cref{prop:mat-unit1} we have the $q$-trace relations
\[
\mathrm{Tr}(\pi(K_{2\rho}^{-1})\mathsf{P})=q^{-(2\rho,\lambda_{m})},\quad\mathrm{Tr}(\pi(K_{2\rho})\mathsf{Q})=q^{(2\rho,\lambda_{m})}.
\]
The second condition is also satisfied under the assumption that $c^m_n = 0$ if $\lambda_m \neq \lambda_n$. Indeed in this case we have from \cref{prop:conjugation1} that the automorphism $\theta$ is implemented by
\[
\theta (\mathsf{P}) = \pi(K_{2 \rho}^{-1}) \mathsf{P} \pi(K_{2 \rho}), \quad
\theta (\mathsf{Q}) = \pi(K_{2 \rho}) \mathsf{Q} \pi(K_{2 \rho}^{-1}).
\]
Therefore we can apply \cref{lem:hoch-lemma} by setting $V = \pi(K_{2 \rho}^{-1})$ in the case of $\mathsf{P}$ and by setting $V = \pi(K_{2 \rho})$ in the case of $\mathsf{Q}$.
In both cases the twist is given by $\theta$.
\end{proof}

By construction these classes descend to the appropriate full flag manifolds.

\begin{corollary}
\label{cor:class-flag}
With $\mathsf{P}, \mathsf{Q}$ as above we have
\[
[C(\mathsf{P})] \in HH^\theta_2 (\mathbb{C}_q[T \backslash G]), \quad
[C(\mathsf{Q})] \in HH^\theta_2 (\mathbb{C}_q[G / T]).
\]
\end{corollary}

\begin{proof}
Under our assumptions on the coefficients $c^m_n$, it follows from \cref{prop:proj-flag} that $\mathsf{P}^i_j \in \mathbb{C}_q[T \backslash G]$ and $\mathsf{Q}^i_j \in \mathbb{C}_q[G/T]$.
The conclusion then follows.
\end{proof}

The rest of the paper will be devoted to proving non-triviality of these classes. The strategy will be to define some appropriate twisted 2-cocycles and to show that their pairings are non-zero in most cases.
A word of warning before proceeding: we will prove non-triviality of the class $[C(\mathsf{P})]$ in $HH^\theta_2 (\mathbb{C}_q[T \backslash G])$ and of the class $[C(\mathsf{Q})]$ in $HH^\theta_2 (\mathbb{C}_q[G / T])$, but we will leave open the question of non-triviality of these classes in $HH^\theta_2 (\mathbb{C}_q[G])$.

\subsection{Twisted 2-cocycles}

We now turn to twisted $2$-cocycles.
We start by recalling some properties satisfied by the counit, which will be needed for the definition of the cocycles.

\begin{lemma}
\label{lem:counit}
The counit $\varepsilon : \mathbb{C}_q[G] \to \mathbb{C}$ satisfies the following properties:

1) for any $X \in U_q(\mathfrak{g})$ and $a \in \mathbb{C}_q[G]$ we have $\varepsilon(X\triangleright a)=\varepsilon(a\triangleleft X)$.

2) the restriction $\varepsilon : \mathbb{C}_q[G/T] \to \mathbb{C}$ is invariant under $\sigma_{\lambda, \lambda^\prime}$, that is $\varepsilon \circ \sigma_{\lambda, \lambda^\prime} = \varepsilon$,

3) the restriction $\varepsilon : \mathbb{C}_q[T \backslash G] \to \mathbb{C}$ is invariant under $\sigma_{\lambda, \lambda^\prime}$, that is $\varepsilon \circ \sigma_{\lambda, \lambda^\prime} = \varepsilon$.
\end{lemma}

\begin{proof}
1) Recall that the left and right actions are defined by
\[
(Y \triangleright \phi)(X) = \phi(X Y), \quad
(\phi \triangleleft Y)(X) = \phi(Y X).
\]
The counit is defined by $\varepsilon(\phi) = \phi(1)$. Hence we obtain
\[
\varepsilon(Y \triangleright \phi) = (Y \triangleright \phi)(1) = \phi(Y) = (\phi \triangleleft Y)(1) = \varepsilon(\phi \triangleleft Y).
\]

2) We have to show that $\varepsilon(\sigma_{\lambda, \lambda^\prime}(a)) = \varepsilon(a)$
for all $a \in \mathbb{C}_q[G/T]$. 
Using 1) we get
\[
\varepsilon(\sigma_{\lambda, \lambda^\prime}(a)) = \varepsilon(K_\lambda \triangleright a \triangleleft K_{\lambda^\prime}) = \varepsilon(K_\lambda K_{\lambda^\prime} \triangleright a).
\]
Finally we have $K_\lambda K_{\lambda^\prime} \triangleright a = a$, since $a \in \mathbb{C}_q[G/T]$, 
which shows the invariance.

3) The proof is completely analogous to that of 2).
\end{proof}

Next we have a simple identity for the action of the generators $E_i$ and $F_i$ under the counit.

\begin{lemma}
\label{lem:counit-action}
Let $X = E_i, F_i$ be one of the generators of $U_q(\mathfrak{g})$. Then:

1) we have $\varepsilon(X \triangleright (a b)) = \varepsilon(X \triangleright a) \varepsilon(b) + \varepsilon(a) \varepsilon(X \triangleright b)$ for all $a, b \in \mathbb{C}_q[G/T]$,

2) we have $\varepsilon(X \triangleright (a b)) = \varepsilon(X \triangleright a) \varepsilon(b) + \varepsilon(a) \varepsilon(X \triangleright b)$ for all $a, b \in \mathbb{C}_q[T\backslash G]$.
\end{lemma}

\begin{proof}
Recall that in general for all $X \in U_q(\mathfrak{g})$ and $a, b \in \mathbb{C}_q[G]$ we have
\[
X \triangleright (a b) = (X_{(1)} \triangleright a) (X_{(2)} \triangleright b),
\quad
(a b) \triangleleft X = (a \triangleleft X_{(1)}) (b  \triangleleft X_{(2)}).
\]

1) We will consider $X = E_i$, the other case being identical.
For $a, b \in \mathbb{C}_q[G/T]$ we have
\[
E_i \triangleright (a b) = (E_i \triangleright a) (K_i \triangleright b) + a (E_i \triangleright b)
= (E_i \triangleright a) b + a (E_i \triangleright b),
\]
where we have used the fact that $K_\lambda \triangleright a = a$ for all $a \in \mathbb{C}_q[G/T]$.
Since the counit is a homomorphism we obtain the result.

2) For $a, b \in \mathbb{C}_q[T \backslash G]$ we can proceed as above. Using the fact that $a \triangleleft K_\lambda = a$ for all $a \in \mathbb{C}_q[T \backslash G]$ we easily obtain the identity
\[
\varepsilon((a b) \triangleleft X) = \varepsilon(a  \triangleleft X) \varepsilon(b) + \varepsilon(a) \varepsilon(b  \triangleleft X).
\]
But from \cref{lem:counit} we have $\varepsilon(a \triangleleft X) = \varepsilon(X \triangleright a)$, hence we obtain the same expression.
\end{proof}

We are now ready to define some twisted $2$-cocycles.

\begin{proposition}
Let $X = E_i, F_i$ and $Y = E_j, F_j$ be some of the generators of $U_q(\mathfrak{g})$.
Define the linear functional $\eta_{X, Y} : \mathbb{C}_q[G]^{\otimes 3} \to \mathbb{C}$ by the formula
\[
\eta_{X, Y} (a_0 \otimes a_1 \otimes a_2) := \varepsilon(a_0) \varepsilon(X \triangleright a_1) \varepsilon(Y \triangleright a_2).
\]

1) The restriction to $\mathbb{C}_q[G / T]$ gives a cohomology class $[\eta_{X, Y}] \in HH_{\sigma_{\lambda, \lambda^\prime}}^2 (\mathbb{C}_q[G / T])$.

2) The restriction to $\mathbb{C}_q[T \backslash G]$ gives a cohomology class $[\eta_{X, Y}] \in HH_{\sigma_{\lambda, \lambda^\prime}}^2 (\mathbb{C}_q[T \backslash G])$.
\end{proposition}

\begin{proof}
1) We have to show that twisted Hochschild differential applied to the restriction of the functional $\eta_{X, Y}$ gives zero. Using the definition of $\mathrm{b}_{\sigma_{\lambda, \lambda^\prime}}$ we get
\[
\begin{split}
(\mathrm{b}_{\sigma_{\lambda, \lambda^\prime}} \eta_{X, Y}) (a_0 \otimes a_1 \otimes a_2 \otimes a_3)
& = \varepsilon(a_0 a_1) \varepsilon(X \triangleright a_2) \varepsilon(Y \triangleright a_3) -\varepsilon(a_0) \varepsilon(X\triangleright(a_1 a_2)) \varepsilon(Y\triangleright a_3) \\
 & + \varepsilon(a_0) \varepsilon(X\triangleright a_1) \varepsilon(Y\triangleright(a_2 a_3))-\varepsilon(\sigma_{\lambda, \lambda^\prime}(a_3) a_0) \varepsilon(X\triangleright a_1) \varepsilon(Y\triangleright a_2).
\end{split}
\]
For $a_1, a_2 \in \mathbb{C}_q[G / T]$ we have the identity $\varepsilon(X \triangleright (a_1 a_2)) = \varepsilon(X \triangleright a_1) \varepsilon(a_2) + \varepsilon(a_1) \varepsilon(X \triangleright a_2)$ by \cref{lem:counit-action}. Similarly for $Y$.
Then this expression simplifies to
\[
\begin{split}
(\mathrm{b}_{\sigma_{\lambda, \lambda^\prime}} \eta_{X, Y}) (a_{0}\otimes a_{1}\otimes a_{2}\otimes a_{3}) & = \varepsilon(a_{0})\varepsilon(X\triangleright a_{1})\varepsilon(Y\triangleright a_{2})\varepsilon(a_{3}) \\
& - \varepsilon(\sigma_{\lambda,\lambda^{\prime}}(a_{3}))\varepsilon(a_{0})\varepsilon(X\triangleright a_{1})\varepsilon(Y\triangleright a_{2}).
\end{split}
\]
Finally we use the fact that $\varepsilon \circ \sigma_{\lambda, \lambda^\prime} = \varepsilon$ on $\mathbb{C}_q[G / T]$, as shown in \cref{lem:counit}.
Then the two terms cancel out and we conclude that $\mathrm{b}_{\sigma_{\lambda, \lambda^\prime}} \eta_{X, Y} = 0$.

2) The proof is completely identical to that of 1), thanks to \cref{lem:counit-action}.
\end{proof}

\begin{remark}
We do not obtain classes in $HH_{\sigma_{\lambda, \lambda^\prime}}^2 (\mathbb{C}_q[G])$
in this way. One of the reasons is that the counit fails to be invariant under the automorphism $\sigma_{\lambda, \lambda^\prime}$ on $\mathbb{C}_q[G]$.
\end{remark}

In the following we will also use the notation
\[
\eta_{a} (a_0 \otimes a_1 \otimes a_2) := \eta_{F_a, E_a} (a_0 \otimes a_1 \otimes a_2) = \varepsilon(a_{0})\varepsilon(F_{a}\triangleright a_{1})\varepsilon(E_{a}\triangleright a_{2}).
\]

\section{Computation of the pairings}
\label{sec:pairing}

In this section we will compute the pairings $\eta_a(C(\mathsf{P}))$ and $\eta_a(C(\mathsf{Q}))$.
Since this computation will be somewhat lengthy, we will split it into several subsections.

\subsection{Some simplifications}

We start by proving some useful lemmata that will be needed to compute the pairings.
First we look at the expression for $\eta_{X, Y}(C(\mathsf{P}))$.

\begin{lemma}
\label{lem:simplifications}
We have the formula
\[
\eta_{X, Y} (C(\mathsf{P})) = \sum_{i, j, k} q^{-(2\rho, \lambda_i)} (2c^i_j - \delta^i_j) \varepsilon(X \triangleright\mathsf{P}^j_k) \varepsilon(Y \triangleright \mathsf{P}^k_i).
\]
\end{lemma}

\begin{proof}
Recall that $C(\mathsf{P}) = \mathrm{Tr} \left( \pi(K_{2 \rho}^{-1}) (2 \mathsf{P} - \mathrm{Id}) \otimes \mathsf{P} \otimes \mathsf{P} \right)$.
Writing the trace map explicitly and plugging this expression into $\eta_{X, Y}$ we get
\[
\eta_{X, Y} (C(\mathsf{P})) = \sum_{i, j, k, \ell} \pi(K_{2 \rho}^{-1})^i_j (2 \varepsilon(\mathsf{P}^j_k) - \delta^j_k) \varepsilon(X \triangleright \mathsf{P}^k_\ell) \varepsilon(Y \triangleright \mathsf{P}^{\ell}_i).
\]
From $\varepsilon(\mathsf{M}^n_m)^i_j = \delta^i_m \delta^n_j$ we get $\varepsilon(\mathsf{P}_{k}^{j}) = c_{k}^{j}$.
Moreover we have $\pi(K_{2\rho}^{-1})_{j}^{i}=q^{-(2\rho,\lambda_{i})}\delta_{j}^{i}$.
\end{proof}

For the purpose of computing the pairing $\eta_a(C(\mathsf{Q}))$, it will be useful to consider a generalization of the above expression.
This is given in the next definition.

\begin{notation}
For $X,Y\in U_{q}(\mathfrak{g})$ and any weight $\lambda$ we define
\[
\eta_{X, Y}^\lambda(\mathsf{P}) := \sum_{i,j,k}q^{(\lambda,\lambda_{i})}(2c_{j}^{i}-\delta_{j}^{i})\varepsilon(X\triangleright\mathsf{P}_{k}^{j})\varepsilon(Y\triangleright\mathsf{P}_{i}^{k}).
\]
We will also write $\eta_a^\lambda(\mathsf{P}) := \eta_{F_a, E_a}^\lambda(\mathsf{P})$.
\end{notation}

Clearly we have $\eta_{X, Y} (C(\mathsf{P})) = \eta_{X, Y}^{-2 \rho}(\mathsf{P})$.
Next we will write explicitly the action of the elements $X$ and $Y$ on the matrix elements $\mathsf{P}^j_k$ and $\mathsf{P}^k_i$.

\begin{lemma}
\label{lem:pairing-xy}
We have the formula
\[
\eta_{X, Y}^\lambda(\mathsf{P}) = \sum_{i,j,k,\ell,m,n}(2c_{j}^{i}-\delta_{j}^{i})\pi(S(X_{(1)}))_{k}^{j}c_{\ell}^{k}\pi(X_{(2)}S(Y_{(1)}))_{m}^{\ell}c_{n}^{m}\pi(Y_{(2)}K_{\lambda})_{i}^{n}.
\]
\end{lemma}

\begin{proof}
Using $(\mathsf{M}^n_m)^i_j = u^{m*}_i u^n_j$ and the formulae in \eqref{eq:action-ugen} we compute
\[
X \triangleright (\mathsf{M}^n_m)^i_j = (X_{(1)} \triangleright u^{m*}_i) (X_{(2)} \triangleright u^n_j) = \sum_{k, \ell} \pi(S(X_{(1)}))^i_k \pi(X_{(2)})^\ell_j u^{m*}_k u^n_\ell.
\]
Since $\varepsilon(u^i_j) = \varepsilon(u^{i*}_j) = \delta^i_j$ we obtain $\varepsilon(X \triangleright (\mathsf{M}^n_m)^i_j) = \pi(S(X_{(1)}))^i_m \pi(X_{(2)})^n_j$.
Then
\[
\begin{split}
\sum_k \varepsilon(X \triangleright\mathsf{P}_{k}^{j}) \varepsilon(Y\triangleright\mathsf{P}_{i}^{k})
& =\sum_{k,m,n,o,p}c_{n}^{m}c_{p}^{o}\varepsilon(X\triangleright(\mathsf{M}_{m}^{n})_{k}^{j})\varepsilon(Y\triangleright(\mathsf{M}_{o}^{p})_{i}^{k})\\
 & =\sum_{k,m,n,o,p}c_{n}^{m}c_{p}^{o}\pi(S(X_{(1)}))_{m}^{j}\pi(X_{(2)})_{k}^{n}\pi(S(Y_{(1)}))_{o}^{k}\pi(Y_{(2)})_{i}^{p}.
\end{split}
\]
The sum over $k$ can be rewritten as a product of matrices, that
is
\[
\sum_{k}\varepsilon(X\triangleright\mathsf{P}_{k}^{j})\varepsilon(Y\triangleright\mathsf{P}_{i}^{k})=\sum_{m,n,o,p}\pi(S(X_{(1)}))_{m}^{j}c_{n}^{m}\pi(X_{(2)}S(Y_{(1)}))_{o}^{n}c_{p}^{o}\pi(Y_{(2)})_{i}^{p}.
\]
Plugging this back into our expression we obtain
\[
\eta_{X, Y}^\lambda(\mathsf{P}) = \sum_{i,j}q^{(\lambda,\lambda_{i})}(2c_{j}^{i}-\delta_{j}^{i})\sum_{m,n,o,p}\pi(S(X_{(1)}))_{m}^{j}c_{n}^{m}\pi(X_{(2)}S(Y_{(1)}))_{o}^{n}c_{p}^{o}\pi(Y_{(2)})_{i}^{p}.
\]
Finally, since $q^{(\lambda,\lambda_{i})}=\pi(K_{\lambda})_{i}^{i}$ we obtain the result.
\end{proof}

The next lemma assumes the condition on the coefficients $c^m_n$ discussed before.
It will be used to move the Cartan elements $K_\lambda$ across various matrix coefficients.

\begin{lemma}
\label{lem:useful}
Suppose $c^m_n = 0$ if $\lambda_m \neq \lambda_n$. Then for any $X,Y\in U_{q}(\mathfrak{g})$ we have
\[
\pi(X K_\lambda)_{j}^{i} c_{k}^{j} \pi(K_{\lambda^\prime} Y)_{\ell}^{k} = \pi(X K_{\lambda} K_{\lambda^{\prime}})_{j}^{i} c_{k}^{j} \pi(Y)_{\ell}^{k} = \pi(X)_{j}^{i} c_{k}^{j} \pi(K_{\lambda} K_{\lambda^{\prime}}Y)_{\ell}^{k}.
\]
\end{lemma}

\begin{proof}
Since we have $\pi(K_{\lambda})^i_j = \delta^i_j q^{(\lambda, \lambda_i)}$ we can rewrite
\[
\pi(XK_{\lambda})_{j}^{i}c_{k}^{j}\pi(K_{\lambda^{\prime}}Y)_{\ell}^{k}=\pi(X)_{j}^{i}\pi(K_{\lambda})_{j}^{j}c_{k}^{j}\pi(K_{\lambda^{\prime}})_{k}^{k}\pi(Y)_{\ell}^{k}.
\]
Next we have $\pi(K_{\lambda})_{i}^{i}=\pi(K_{\lambda})_{j}^{j}$ for $\lambda_i = \lambda_j$.
Since by assumption $c_{k}^{j}=0$ if $\lambda_{j}\neq\lambda_{k}$, we have the identity $c_{k}^{j}\pi(K_{\lambda^{\prime}})_{k}^{k}=\pi(K_{\lambda^{\prime}})_{j}^{j}c_{k}^{j}$.
Then we obtain
\[
\pi(XK_{\lambda})_{j}^{i}c_{k}^{j}\pi(K_{\lambda^{\prime}}Y)_{\ell}^{k}=\pi(X)_{j}^{i}\pi(K_{\lambda}K_{\lambda^{\prime}})_{j}^{j}c_{k}^{j}\pi(Y)_{\ell}^{k}=\pi(XK_{\lambda}K_{\lambda^{\prime}})_{j}^{i}c_{k}^{j}\pi(Y)_{\ell}^{k}.
\]
Similarly the second equality is obtained by writing $\pi(K_{\lambda})_{j}^{j}c_{k}^{j}=c_{k}^{j}\pi(K_{\lambda})_{k}^{k}$.
\end{proof}

\subsection{Organization of the computation}

Now our aim is to simplify the expression given in \cref{lem:pairing-xy} in the case $X = F_a$ and $Y = E_a$.
Since this expression involves coproducts, it is convenient to introduce the following notation in order to handle the different terms.

\begin{notation}
For $X, X^\prime, Y, Y^\prime \in U_q(\mathfrak{g})$ we define
\[
\Xi^{\lambda}(X\otimes X^{\prime}\otimes Y\otimes Y^{\prime}) := \sum_{i,j,m,n,o,p}(2c_{j}^{i}-\delta_{j}^{i})\pi(X)_{m}^{j}c_{n}^{m}\pi(X^{\prime}Y)_{o}^{n}c_{p}^{o}\pi(Y^{\prime}K_{\lambda})_{i}^{p}.
\]
With this notation we have $\eta_{X, Y}^\lambda(\mathsf{P}) = \Xi^\lambda (S(X_{(1)}) \otimes  X_{(2)} \otimes S(Y_{(1)}) \otimes  Y_{(2)})$.
\end{notation}

The expression $S(X_{(1)}) \otimes  X_{(2)} \otimes S(Y_{(1)}) \otimes  Y_{(2)}$ contains four terms in the case $X = F_a$ and $Y = E_a$.
In our conventions these are explicitly given by
\[
\begin{split}
S(X_{(1)}) \otimes X_{(2)} \otimes S(Y_{(1)}) \otimes Y_{(2)} & = K_{a}F_{a} \otimes 1 \otimes E_{a} K_{a}^{-1} \otimes K_{a} -K_{a }F_{a} \otimes 1 \otimes 1 \otimes E_{a}\\
 & - K_{a} \otimes F_{a} \otimes E_{a} K_{a}^{-1} \otimes K_{a} + K_{a} \otimes F_{a} \otimes 1 \otimes E_{a}.
\end{split}
\]
In the next subsection we will compute the value of the functional $\Xi^{\lambda}$ when applied to these four terms.
This will allow us to obtain a simple expression for $\eta_a (C(\mathsf{P}))$.

\subsection{Computation of the four terms}

We start by computing the functional $\Xi^\lambda$ applied to the first and fourth term in the expansion of
$S(X_{(1)})\otimes X_{(2)}\otimes S(Y_{(1)})\otimes Y_{(2)}$, in the case $X = F_a$ and $Y = E_a$.
The next lemma shows that these take the same values.

\begin{lemma}
\label{lem:computation-14}
We have the identities
\[
\begin{split}
\Xi^\lambda (K_a \otimes F_a\otimes 1 \otimes E_a) & = \Xi^\lambda (K_a F_a \otimes 1 \otimes E_a K_a^{-1} \otimes K_a) \\
 & = \sum_{i, j, k, \ell} c^i_j \pi(K_a F_a)^j_k c^k_\ell \pi(E_a K_\lambda)^\ell_i.
\end{split}
\]
\end{lemma}

\begin{proof}
Let us start by considering the fourth term $K_a \otimes F_a\otimes 1 \otimes E_a$.
We have
\[
\Xi^{\lambda}(K_{a}\otimes F_{a}\otimes1\otimes E_{a})=\sum_{i,j,m,n,o,p}(2c_{j}^{i}-\delta_{j}^{i})\pi(K_{a})_{m}^{j}c_{n}^{m}\pi(F_{a})_{o}^{n}c_{p}^{o}\pi(E_{a}K_{\lambda})_{i}^{p}.
\]
Using \cref{lem:useful} we rewrite this expression as
\[
\Xi^{\lambda}(K_{a}\otimes F_{a}\otimes1\otimes E_{a})=\sum_{i,j,n,o,p}(2c_{j}^{i}-\delta_{j}^{i})c_{n}^{j}\pi(K_{a}F_{a})_{o}^{n}c_{p}^{o}\pi(E_{a}K_{\lambda})_{i}^{p}.
\]
We have $\sum_j (2c^i_j - \delta^i_j) c^j_n = c^i_n$, thanks to the identity $\sum_k c^i_k c^k_j = c^i_j$. Hence
\[
\Xi^{\lambda}(K_{a}\otimes F_{a}\otimes1\otimes E_{a})=\sum_{i,n,o,p}c_{n}^{i}\pi(K_{a}F_{a})_{o}^{n}c_{p}^{o}\pi(E_{a}K_{\lambda})_{i}^{p}.
\]

Now consider the first term $K_{a}F_{a}\otimes1\otimes E_{a}K_{a}^{-1}\otimes K_{a}$.
We have
\[
\Xi^{\lambda}(K_{a}F_{a}\otimes1\otimes E_{a}K_{a}^{-1}\otimes K_{a})=\sum_{i,j,m,n,o,p}(2c_{j}^{i}-\delta_{j}^{i})\pi(K_{a}F_{a})_{m}^{j}c_{n}^{m}\pi(E_{a}K_{a}^{-1})_{o}^{n}c_{p}^{o}\pi(K_{a}K_{\lambda})_{i}^{p}.
\]
Using \cref{lem:useful} we rewrite this expression as
\[
\Xi^{\lambda}(K_{a}F_{a}\otimes1\otimes E_{a}K_{a}^{-1}\otimes K_{a})=\sum_{i,j,m,n,o}(2c_{j}^{i}-\delta_{j}^{i})\pi(K_{a}F_{a})_{m}^{j}c_{n}^{m}\pi(E_{a}K_{\lambda})_{o}^{n}c_{i}^{o}.
\]
Finally using the identity $\sum_{k}c_{k}^{i}c_{j}^{k}=c_{j}^{i}$
this can be rewritten as
\[
\Xi^{\lambda}(K_{a}F_{a}\otimes1\otimes E_{a}K_{a}^{-1}\otimes K_{a})=\sum_{j,m,n,o}c_{j}^{o}\pi(K_{a}F_{a})_{m}^{j}c_{n}^{m}\pi(E_{a}K_{\lambda})_{o}^{n}.
\]
Comparing the two expressions we see that they are identical.
\end{proof}

Next we apply the functional $\Xi^\lambda$ to the the second and third term. The next lemma shows that these take a different form with respect to the previous two terms.

\begin{lemma}
\label{lem:computation-23}
We have the identities
\[
\begin{split}
\Xi^\lambda (K_a F_a \otimes 1 \otimes 1 \otimes E_a) & = 2 \Xi^\lambda (K_a \otimes F_a \otimes 1 \otimes E_a) - \sum_{i,j}c_{j}^{i}\pi(E_{a}K_{\lambda}K_{a}F_{a})_{i}^{j}, \\
\Xi^\lambda (K_a \otimes F_a \otimes  E_a K_a^{-1} \otimes K_a) & = \sum_{i,j}c_{j}^{i}\pi(K_{a}F_{a}E_{a}K_{\lambda})_{i}^{j}.
\end{split}
\]
\end{lemma}

\begin{proof}
Consider the second term $K_{a} F_{a}\otimes 1 \otimes 1 \otimes E_{a}$.
We have
\[
\Xi^\lambda (K_a F_a \otimes 1 \otimes 1 \otimes E_a) = \sum_{i,j,m,n,o,p}(2c_{j}^{i}-\delta_{j}^{i})\pi(K_{a}F_{a})_{m}^{j}c_{n}^{m}\pi(1)_{o}^{n}c_{p}^{o}\pi(E_{a}K_{\lambda})_{i}^{p}.
\]
Using the relation $\sum_{k}c_{k}^{i}c_{j}^{k}=c_{j}^{i}$ this becomes
\[
\Xi^{\lambda} (K_a F_a \otimes 1 \otimes 1 \otimes E_a) = \sum_{i,j,m,p}(2c_{j}^{i}-\delta_{j}^{i})\pi(K_{a}F_{a})_{m}^{j}c_{p}^{m}\pi(E_{a}K_{\lambda})_{i}^{p}.
\]
Moreover we have the following identity
\[
\sum_{i,j,m,p} \delta_{j}^{i}\pi(K_a F_a)_{m}^{j} c_{p}^{m} \pi(E_a K_\lambda)_{i}^{p} = \sum_{m,p} c_{p}^{m} \pi(E_a K_\lambda K_a F_a)_{m}^{p}.
\]
Then comparing with \cref{lem:computation-14} we see that
\[
\Xi^\lambda (K_a F_a \otimes 1 \otimes 1 \otimes E_a) = 2 \Xi^\lambda (K_a \otimes F_a \otimes 1 \otimes E_a) - \sum_{m, p} c_{p}^{m} \pi(E_a K_\lambda K_a F_a)_{m}^{p}.
\]

Next consider the third term $K_{a}\otimes F_{a}\otimes E_{a}K_{a}^{-1}\otimes K_{a}$.
We have
\[
\Xi^{\lambda}(K_{a}\otimes F_{a}\otimes E_{a}K_{a}^{-1}\otimes K_{a})=\sum_{i,j,m,n,o,p}(2c_{j}^{i}-\delta_{j}^{i})\pi(K_{a})_{m}^{j}c_{n}^{m}\pi(F_{a}E_{a}K_{a}^{-1})_{o}^{n}c_{p}^{o}\pi(K_{a}K_{\lambda})_{i}^{p}.
\]
Using \cref{lem:useful} this can be rewritten as
\[
\Xi^{\lambda} (K_{a} \otimes F_{a}\otimes E_{a} K_{a}^{-1} \otimes K_{a}) = \sum_{i,j,n,o} (2c_{j}^{i} - \delta_{j}^{i}) c_{n}^{j} \pi(K_{a} F_{a} E_{a} K_{\lambda})_{o}^{n} c_{i}^{o}.
\]
Finally using the identity $\sum_{k}c_{k}^{i}c_{j}^{k}=c_{j}^{i}$
twice we obtain
\[
\Xi^\lambda (K_a \otimes F_a \otimes  E_a K_a^{-1} \otimes K_a) = \sum_{n, o} c^o_n \pi(K_a F_a E_a K_\lambda)^n_o.
\qedhere
\]
\end{proof}

\subsection{Computation of $\eta_a(C(\mathsf{P}))$}

Now we are in the position to conclude the computation of $\eta_a(C(\mathsf{P}))$.
First we put together all the previous results.

\begin{lemma}
\label{prop:general-comp}
We have the identity
\[
\eta_a^{\lambda}(\mathsf{P}) = \sum_{i,j}c_{j}^{i}\pi(E_{a}K_{\lambda}K_{a}F_{a})_{i}^{j}-\sum_{i,j}c_{j}^{i}\pi(K_{a}F_{a}E_{a}K_{\lambda})_{i}^{j}.
\]
\end{lemma}

\begin{proof}
Applying $\Xi^\lambda$ to $S(X_{(1)})\otimes X_{(2)}\otimes S(Y_{(1)})\otimes Y_{(2)}$ with
$X = F_a$ and $Y = E_a$ we get
\[
\begin{split}
\eta_a^\lambda (\mathsf{P}) & =\Xi^{\lambda}(K_a F_a \otimes 1 \otimes E_a K_a^{-1} \otimes K_a) -\Xi^{\lambda}(K_{a}F_{a}\otimes 1 \otimes 1 \otimes E_{a}) \\
 & -\Xi^{\lambda}(K_{a}\otimes F_{a}\otimes E_{a}K_{a}^{-1} \otimes K_{a}) + \Xi^{\lambda}(K_{a} \otimes F_{a} \otimes 1 \otimes E_{a}).
\end{split}
\]
Combining \cref{lem:computation-14} and \cref{lem:computation-23} we can write
\[
\begin{split}
\Xi^{\lambda} (K_{a} F_{a} \otimes 1 \otimes 1 \otimes E_{a}) & =\Xi^{\lambda}(K_{a} F_{a}\otimes 1 \otimes E_{a} K_{a}^{-1} \otimes K_{a}) + \Xi^{\lambda}(K_{a} \otimes F_{a}\otimes 1 \otimes E_{a}) \\
 & -\sum_{i,j} c_{j}^{i} \pi(E_{a}K_{\lambda} K_{a} F_{a})_{i}^{j}.
\end{split}
\]
Plugging this into $\eta_a^\lambda (\mathsf{P})$ we see that two terms cancel out. Finally using the explicit expression for $\Xi^\lambda (K_a \otimes F_a \otimes E_a K_a^{-1} \otimes K_a)$ we conclude that
\[
\eta_{a}^{\lambda}(\mathsf{P})=\sum_{i,j}c_{j}^{i}\pi(E_{a}K_{\lambda}K_{a}F_{a})_{i}^{j}-\sum_{i,j}c_{j}^{i}\pi(K_{a}F_{a}E_{a}K_{\lambda})_{i}^{j}.
\qedhere
\]
\end{proof}

Now we specialize to the case $\lambda = -2 \rho$, corresponding to the pairing $\eta_a (C(\mathsf{P}))$.
In this situation we can make a further simplification, which gives a very simple result.

\begin{proposition}
\label{thm:pairingP}
Let $\mathsf{P} = \sum_{m, n} c^m_n \mathsf{M}^n_m$ be a projection with $c^m_n = 0$ if $\lambda_m \neq \lambda_n$. Then
\[
\eta_a (C(\mathsf{P})) = \sum_i c^i_i q^{(\alpha_a - 2\rho, \lambda_i)} [d_a^{-1} (\alpha_a, \lambda_i)]_{q_a}.
\]
\end{proposition}

\begin{proof}
Recall the commutation relations $E_{a}K_{\lambda}=q^{-(\alpha_{a},\lambda)}K_{\lambda}E_{a}$
and $F_{a}K_{\lambda}=q^{(\alpha_{a},\lambda)}K_{\lambda}F_{a}$.
From these we immediately derive $F_{a}E_{a}K_{2\rho}^{-1}=K_{2\rho}^{-1}F_{a}E_{a}$.
A less obvious identity is
\[
E_{a}K_{2\rho}^{-1}K_{a}=K_{2\rho}^{-1}K_{a}E_{a}.
\]
This can be seen as follows. We have $E_{a}K_{2\rho}^{-1}K_{a}=q^{(2\rho-\alpha_{a},\alpha_{a})}K_{2\rho}^{-1}K_{a}E_{a}$ from the commutation relations.
Next we show that $(2\rho,\alpha_{a})=(\alpha_{a},\alpha_{a})$. Recall that $\rho$ can be written as $\rho=\sum_{i}\omega_{i}$,
where $\{\omega_{i}\}_{i}$ are the fundamental weights. Then we have
\[
(2 \rho, \alpha_a) = (\alpha_a, \alpha_a) \sum_i (\omega_i, \alpha_a^\vee) = (\alpha_a, \alpha_a) \sum_i \delta_{i a} = (\alpha_a, \alpha_a),
\]
where we have used that the fundamental weights are dual to the coroots $\alpha_a^\vee = 2 \alpha_a / (\alpha_a, \alpha_a)$.

Using the commutation relations above we can rewrite \cref{prop:general-comp} in the form
\[
\eta_{a}(C(\mathsf{P}))=\sum_{i,j}c_{j}^{i}\pi(K_{2\rho}^{-1}K_{a}[E_{a},F_{a}])_{i}^{j}.
\]
Now we can use the commutation relations $[E_a, F_a] = \frac{K_a - K_a^{-1}}{q_a - q_a^{-1}}$. Then
\[
\eta_{a}(C(\mathsf{P})) = \sum_{i,j}c_{j}^{i} \pi \left( K_{2\rho}^{-1}K_{a}\frac{K_{a}-K_{a}^{-1}}{q_{a}-q_{a}^{-1}} \right)_{i}^{j}.
\]
Next we have $\pi(K_\lambda)^i_j = \delta^i_j q^{(\lambda, \lambda_i)}$, where $\{\lambda_i\}_i$ are the weights corresponding to our choice of basis for $V(\Lambda)$.
Then the above expression can be rewritten as
\[
\eta_a (C(\mathsf{P})) = \sum_{i}c_{i}^{i} q^{(\alpha_{a}-2\rho,\lambda_{i})} \frac{q^{(\alpha_{a},\lambda_{i})} - q^{-(\alpha_{a},\lambda_{i})}}{q_{a}-q_{a}^{-1}}.
\]
Finally since $q_{a}=q^{d_{a}}$ we have the identity $[d_{a}^{-1}(\alpha_a, \lambda_i)]_{q_a} = 
\frac{q^{(\alpha_a, \lambda_i)} - q^{-(\alpha_a, \lambda_i)}}{q_a - q_a^{-1}}$.
\end{proof}

\subsection{Computation of $\eta_a(C(\mathsf{Q}))$}

The computation of the pairing $\eta_a(C(\mathsf{Q}))$ can be essentially reduced to that of $\eta_a(C(\mathsf{P}))$.
To see this we need the following simple lemma.

\begin{lemma}
\label{lem:EF-action-N}
Suppose $\lambda_{m}=\lambda_{n}$. Then we have
\[
\begin{split}
\varepsilon(E_a \triangleright (\mathsf{N}^n_m)^i_j)
& = - q^{-(\alpha_a, \lambda_j)} \varepsilon(E_a \triangleright (\mathsf{M}^n_m)^i_j), \\
\varepsilon(F_a \triangleright (\mathsf{N}^n_m)^i_j)
& = - q^{-(\alpha_a, \lambda_i)}\varepsilon(F_a \triangleright (\mathsf{M}^n_m)^i_j).
\end{split}
\]
\end{lemma}

\begin{proof}
We have seen in the proof of \cref{lem:pairing-xy} that $\varepsilon(X \triangleright (\mathsf{M}^n_m)^i_j) = \pi(S(X_{(1)}))^i_m \pi(X_{(2)})^n_j$.
Similarly we obtain the expression $ \varepsilon(X \triangleright (\mathsf{N}^n_m)^i_j) = \pi(X_{(1)})^i_m \pi(S(X_{(2)}))^n_j$.

Now consider the case $X = E_a$. Then we compute
\[
\begin{split}
\varepsilon(E_{a}\triangleright(\mathsf{M}_{m}^{n})_{j}^{i}) & =-\pi(E_{a}K_{a}^{-1})_{m}^{i}\pi(K_{a})_{j}^{n}+\pi(1)_{m}^{i}\pi(E_{a})_{j}^{n}\\
 & =-\pi(E_{a})_{m}^{i}\pi(1)_{j}^{n}+\pi(1)_{m}^{i}\pi(E_{a})_{j}^{n},
\end{split}
\]
where in the second line we have used \cref{lem:useful}, since $\lambda_m = \lambda_n$. On the
other hand we have
\[
\varepsilon(E_{a}\triangleright(\mathsf{N}_{m}^{n})_{j}^{i})=\pi(E_{a})_{m}^{i}\pi(K_{a}^{-1})_{j}^{n}-\pi(1)_{m}^{i}\pi(E_{a}K_{a}^{-1})_{j}^{n}.
\]
Comparing the two expressions we get $\varepsilon(E_a \triangleright (\mathsf{N}^n_m)^i_j) = - q^{-(\alpha_a, \lambda_j)} \varepsilon(E_a \triangleright (\mathsf{M}^n_m)^i_j)$.

Similarly consider the case $X = F_a$. We have
\[
\varepsilon(F_{a}\triangleright(\mathsf{M}_{m}^{n})_{j}^{i})=-\pi(K_{a}F_{a})_{m}^{i}\pi(1)_{j}^{n}+\pi(K_{a})_{m}^{i}\pi(F_{a})_{j}^{n}.
\]
On the other hand we compute
\[
\begin{split}\varepsilon(F_{a}\triangleright(\mathsf{N}_{m}^{n})_{j}^{i}) & =\pi(F_{a})_{m}^{i}\pi(1)_{j}^{n}-\pi(K_{a}^{-1})_{m}^{i}\pi(K_{a}F_{a})_{j}^{n}\\
 & =\pi(F_{a})_{m}^{i}\pi(1)_{j}^{n}-\pi(1)_{m}^{i}\pi(F_{a})_{j}^{n},
\end{split}
\]
where we have used \cref{lem:useful} again.
Comparing the two expressions we get the identity $\varepsilon(F_a \triangleright (\mathsf{N}^n_m)^i_j) = - q^{-(\alpha_a, \lambda_i)}\varepsilon(F_a \triangleright (\mathsf{M}^n_m)^i_j)$, which concludes the proof.
\end{proof}

Now we are in the position to compute the pairing $\eta_a(C(\mathsf{Q}))$.

\begin{proposition}
\label{thm:pairingQ}
Let $\mathsf{Q} = \sum_{m, n} c^m_n \mathsf{N}^n_m$ be a projection with $c^m_n = 0$ if $\lambda_m \neq \lambda_n$. Then
\[
\eta_{a}(C(\mathsf{Q}))=\sum_{i}c_{i}^{i}q^{-(\alpha_{a}-2\rho,\lambda_{i})}[d_{a}^{-1}(\alpha_{a},\lambda_{i})]_{q_{a}}.
\]
\end{proposition}

\begin{proof}
Proceeding as in \cref{lem:simplifications} we obtain the formula
\[
\eta_{a}(C(\mathsf{Q}))=\sum_{i,j,k}q^{(2\rho,\lambda_{i})}(2c_{j}^{i}-\delta_{j}^{i})\varepsilon(F_{a}\triangleright\mathsf{Q}_{k}^{j})\varepsilon(E_{a}\triangleright\mathsf{Q}_{i}^{k}).
\]
We start by focusing on the expression
\[
\varepsilon(F_{a}\triangleright\mathsf{Q}_{k}^{j})\varepsilon(E_{a}\triangleright\mathsf{Q}_{i}^{k})=\sum_{m,n,o,p}c_{n}^{m}c_{p}^{o}\varepsilon(F_{a}\triangleright(\mathsf{N}_{m}^{n})_{k}^{j})\varepsilon(E_{a}\triangleright(\mathsf{N}_{o}^{p})_{i}^{k}).
\]
We have $c^m_n = 0$ for $\lambda_{m}\neq\lambda_{n}$ by assumption, hence we can consider $\lambda_m =\lambda_n$ and $\lambda_o = \lambda_p$ in the above expression without loss of generality.
Then we can use \cref{lem:EF-action-N} to get
\[
\begin{split}
\varepsilon(F_{a}\triangleright\mathsf{Q}_{k}^{j})\varepsilon(E_{a}\triangleright\mathsf{Q}_{i}^{k}) & =\sum_{m,n,o,p}c_{n}^{m}c_{p}^{o} q^{-(\alpha_{a}, \lambda_i + \lambda_j)} \varepsilon(F_{a}\triangleright(\mathsf{M}_{m}^{n})_{k}^{j})\varepsilon(E_{a}\triangleright(\mathsf{M}_{o}^{p})_{i}^{k}) \\
 & =q^{-(\alpha_{a},\lambda_{i}+\lambda_{j})}\varepsilon(F_{a}\triangleright\mathsf{P}_{k}^{j})\varepsilon(E_{a}\triangleright\mathsf{P}_{i}^{k}).
\end{split}
\]
We can also assume $\lambda_{i}=\lambda_{j}$, since we multiply this expression by $2c_{j}^{i}-\delta_{j}^{i}$. Then 
\[
\eta_{a}(C(\mathsf{Q}))=\sum_{i,j,k}q^{(2\rho-2\alpha_{a},\lambda_{i})}(2c_{j}^{i}-\delta_{j}^{i})\varepsilon(F_{a}\triangleright\mathsf{P}_{k}^{j})\varepsilon(E_{a}\triangleright\mathsf{P}_{i}^{k}).
\]
Therefore we have obtained the equality $\eta_a(C(\mathsf{Q})) = \eta_a^\lambda(\mathsf{P})$
with $\lambda = 2 \rho - 2 \alpha_a$.
Now we can use \cref{prop:general-comp} with $K_\lambda = K_{2 \rho} K_a^{-2}$. We find the expression
\[
\eta_{a}(C(\mathsf{Q}))=\sum_{i,j}c_{j}^{i}\pi(E_{a}K_{2\rho}K_{a}^{-1}F_{a})_{i}^{j}-\sum_{i,j}c_{j}^{i}\pi(K_{a}F_{a}E_{a}K_{2\rho}K_{a}^{-2})_{i}^{j}.
\]
To proceed we use the commutation relations. In general we have $K_{\lambda}F_{a}E_{a}=F_{a}E_{a}K_{\lambda}$.
Moreover we have seen in a previous computation that $E_{a}K_{2\rho}K_{a}^{-1}=K_{2\rho}K_{a}^{-1}E_{a}$.
Then
\[
\eta_{a}(C(\mathsf{Q}))=\sum_{i,j}c_{j}^{i}\pi(K_{2\rho}K_{a}^{-1}[E_{a},F_{a}])_{i}^{j}.
\]
Finally we proceed as for $\eta_{a}(C(\mathsf{P}))$ to obtain the expression in the theorem.
\end{proof}

\section{Non-triviality and linear independence}
\label{sec:classes}

In this section we will give some more precise statements regarding non-triviality of the classes obtained in the previous sections.
We will also show that the twisted Hochschild homology groups $HH_2^\theta(\flagL)$ and $HH_2^\theta(\flagR)$ are of dimension at least $\mathrm{rank}(\lieg)$.

\subsection{Non-trivial classes}

We begin by summarizing the results of the previous sections in the theorem below, which gives some sufficient conditions for the non-triviality of the classes $[C(\mathsf{P})]$ and $[C(\mathsf{Q})]$ defined in \cref{prop:twisted-cycles}.
First we introduce some notation.

\begin{notation}
Given an element $\mathsf{P} = \sum_{m, n} c^m_n \mathsf{M}^n_m$ we define
\[
\chi_a(\mathsf{P}) := \sum_i c^i_i q^{(\alpha_a - 2\rho, \lambda_i)} [d_a^{-1} (\alpha_a, \lambda_i)]_{q_a}.
\]
Similarly, given an element $\mathsf{Q} = \sum_{m, n} c^m_n \mathsf{N}^n_m$ we define
\[
\tilde{\chi}_a(\mathsf{Q}) := \sum_i c^i_i q^{-(\alpha_a - 2\rho, \lambda_i)} [d_a^{-1} (\alpha_a, \lambda_i)]_{q_a}.
\]
\end{notation}

\begin{theorem}
\label{thm:conditions}
Let $\mathsf{P}, \mathsf{Q}$ be projections satisfying the condition $c^m_n = 0$ if $\lambda_m \neq \lambda_n$.

1) Suppose $\chi_a(\mathsf{P}) \neq 0$ for some $a$. Then $[C(\mathsf{P})] \in HH^\theta_2 (\flagR)$ is non-trivial.

2) Suppose $\tilde{\chi}_a(\mathsf{Q}) \neq 0$ for some $a$. Then $[C(\mathsf{Q})] \in HH^\theta_2 (\flagL)$ is non-trivial.
\end{theorem}

\begin{proof}
Under the stated assumptions for $\mathsf{P}$ and $\mathsf{Q}$ we have $\chi_a(\mathsf{P}) = \eta_a(C(\mathsf{P}))$ by \cref{thm:pairingP} and $\tilde{\chi}_a(\mathsf{Q}) = \eta_a(C(\mathsf{Q}))$ by \cref{thm:pairingQ}.
The conclusion follows immediately.
\end{proof}

It is worth pointing out that these conditions are quite explicit and therefore easy to check, since they only involve representation-theoretic data.
We see from the conditions that the classes will be generically non-trivial if we consider elements of non-zero weight.

As an important example, we can take the basic projections $\mathsf{P} = \mathsf{M}^m_m$ and $\mathsf{Q} = \mathsf{N}^m_m$ for some $m$.
This will show that the twisted homology groups are non-zero.

\begin{corollary}
Let $\lambda_m$ be a non-zero weight. Then the classes $[C(\mathsf{M}^m_m)] \in HH_2^\theta(\flagR)$ and $[C(\mathsf{N}^m_m)] \in HH_2^\theta(\flagL)$ are non-trivial.
\end{corollary}

\begin{proof}
Let us look at the number $\chi_a(\mathsf{M}^m_m) = q^{(\alpha_a - 2\rho, \lambda_m)} [d_a^{-1} (\alpha_a, \lambda_m)]_{q_a}$. Since $\lambda_m \neq 0$, we can always find a simple root $\alpha_a$ such that $(\alpha_a, \lambda_m) \neq 0$ by non-degeneracy.
Therefore the above number is non-zero and from \cref{thm:conditions} we conclude that $[C(\mathsf{M}^m_m)]$ is non-trivial.

The argument for the class $[C(\mathsf{N}^m_m)]$ is identical and we omit it.
\end{proof}

\begin{remark}
We are not able to conclude whether the case $\lambda_m = 0$ is trivial or not.
\end{remark}

Observe that, since we can define the projections $\mathsf{P} = \mathsf{M}^m_m$ and $\mathsf{Q} = \mathsf{N}^m_m$ for any irreducible representation $V(\Lambda)$, we obtain in this way infinitely many non-trivial classes $[C(\mathsf{M}^m_m)]$ and $[C(\mathsf{N}^m_m)]$.
This naturally leads to the problem of studying their linear independence.

\subsection{Linear independence}

In this subsection we will partially discuss the linear independence of the classes obtained above.
The result will be that the twisted homology groups $HH_2^\theta(\flagL)$ and $HH_2^\theta(\flagR)$ are of dimension at least $\mathrm{rank}(\lieg)$.

First let us see what happens in the case of the quantum $2$-sphere.

\begin{example}
Let $\mathfrak{g} = \mathfrak{sl}(2)$. The corresponding full flag manifold is the quantum $2$-sphere.
Denote by $\alpha$ the unique simple root and by $\omega$ the unique fundamental weight. We have $\omega = \rho = \frac{1}{2} \alpha$. The irreducible representations have highest weight $\Lambda = n \omega$ with $n \in \mathbb{N}$, dimension $n + 1$ and weights given by $- \frac{n}{2} \alpha, \cdots, \frac{n}{2} \alpha$. Write $\lambda_k =  \frac{k}{2} \alpha$ and denote by $\mathsf{P}_k$ the projection corresponding to weight $\lambda_k$.
Then we easily compute
\[
\eta(C(\mathsf{P}_k)) = [(\alpha, \lambda_k)]_q = [k]_q.
\]
Hadfield has shown in \cite{twisted-s2} that the space of twisted $2$-cycles is $1$-dimensional. Let us denote by $\mathsf{P}$ the projection corresponding to the weight $\omega = \frac{1}{2} \alpha$. Then it easily follows from the previous computation that $[C(\mathsf{P}_k)] = [k]_q [C(\mathsf{P})]$.
\end{example}

From the previous example, we can expect that the space of twisted $2$-cycles will have dimension larger than one if $\mathrm{rank}(\lieg) > 1$.
This is indeed the case, as we now show.

\begin{theorem}
\label{thm:dimension}
The twisted Hochschild homology groups $HH_2^\theta(\flagL)$ and $HH_2^\theta(\flagR)$ have dimension at least $\mathrm{rank}(\lieg)$.
\end{theorem}

\begin{proof}
Denote by $\mathsf{M}(\omega_i)_m^n$ the matrix units corresponding to the representation $V(\omega_i)$, where $\{ \omega_i \}_{i = 1}^r$ are the fundamental weights of $\lieg$.
We set $\mathsf{P}_i = \mathsf{M}(\omega_i)_1^1$ for $i = 1, \cdots, r$, where $v_1$ is the normalized highest weight vector of $V(\omega_i)$.
By \cref{prop:proj-flag} these projections descend to the quantum full flag manifold $\flagR$.
The classes $C(\mathsf{P}_1), \cdots, C(\mathsf{P}_r)$ are non-trivial, since
\[
\eta_a(C(\mathsf{P}_i)) = \chi_a (\mathsf{P}_i) = q^{(\alpha_a - 2 \rho, \omega_i)} [d_a^{-1} (\alpha_a, \omega_i)]_{q_a} = \delta_{i a} q^{(\alpha_a - 2 \rho, \omega_a)}.
\]
Here we have used that $(\omega_i, \alpha_j) = \delta_{i j} d_j$.
Hence $[C(\mathsf{P}_i)] \neq 0$, as we observed before.

Now we show that the classes $C(\mathsf{P}_1), \cdots, C(\mathsf{P}_r)$ are linearly independent.
Suppose that $\sum_{i = 1}^r b_i C(\mathsf{P}_i) = 0$ for some $b_i$'s. Then we compute
\[
\eta_a \left( \sum_{i = 1}^r b_i C(\mathsf{P}_i) \right) = \sum_{i = 1}^r b_i \chi_a (\mathsf{P}_i) = b_a q^{(\alpha_a - 2 \rho, \omega_a)}.
\]
This implies that $b_a = 0$. Letting $a$ range between $1$ and $r$ gives the claim.

The claim for $HH_2^\theta(\flagL)$ is proven similarly, using $\mathsf{Q}_i = \mathsf{N}(\omega_i)_1^1$ for $i = 1, \cdots, r$.
\end{proof}

\section{Generalized flag manifolds}
\label{sec:generalized}

In this section we will extend some of the results we have obtained to the case of quantum generalized flag manifolds.
This class of spaces contains all the full flag manifolds.
The main issue to discuss is when the projections $\mathsf{P}$ and $\mathsf{Q}$ descend to the appropriate generalized flag manifolds.
We will give a necessary condition for this to happen, but will not discuss the problem in full generality.
On the other hand we will provide an explicit and interesting example of this setting, namely that of the quantum Grassmannians.

\subsection{Equivariant maps}

We start with some simple results on the action of $U_q(\mathfrak{g})$.
Recall that, given a $U_q(\mathfrak{g})$-module $V$ with action $\triangleright$, we can make $V^*$ into a $U_q(\mathfrak{g})$-module by defining $(X\triangleright f)(v)=f(S(X)\triangleright v)$.
It is convenient to define corresponding right actions.

\begin{notation}
Let $V$ be a $U_q(\mathfrak{g})$-module.
Then we define right actions of $\Uqg$ on $V$ and $V^*$ as follows.
For $v \in V$, $f \in V^*$ and $X \in U_q(\mathfrak{g})$ we set
\[
v \triangleleft X := S(X) \triangleright v, \quad
(f \triangleleft X)(v) := f(X \triangleright v).
\]
\end{notation}

Recall that $\Cqg$ has a canonical $\Uqg$-bimodule structure.
We will look at maps from a $\Uqg$-module $V$ to $\Cqg$ which are equivariant with respect to these actions.

\begin{definition}
We say that a map $\psi : V \to \mathbb{C}_q[G]$ is $\triangleright$-equivariant (respectively $\triangleleft$-equivariant) if $X \triangleright \psi(v) = \psi(X \triangleright v)$ (respectively $\psi(v) \triangleleft X = \psi(v \triangleleft X)$) for all $v \in V$ and $X \in \Uqg$.
\end{definition}

With these definitions, we have the following easy result on matrix coefficients.

\begin{proposition}
\label{prop:equiv-maps}
Let $c^\Lambda_{f, v}$ denote the matrix coefficients of a representation $V(\Lambda)$. Then:

1) the map $V(\Lambda) \to \mathbb{C}_q[G]$ given by $v\mapsto c_{f,v}^{\Lambda}$ is $\triangleright$-equivariant,

2) the map $V(\Lambda)^* \to \mathbb{C}_q[G]$ given by $f\mapsto c_{f,v}^{\Lambda}$ is $\triangleleft$-equivariant,

3) the map $V(\Lambda) \to \mathbb{C}_q[G]$ given by $v\mapsto S(c_{f,v}^{\Lambda})$ is $\triangleleft$-equivariant,

4) the map $V(\Lambda)^* \to \mathbb{C}_q[G]$ given by $f\mapsto S(c_{f,v}^{\Lambda})$ is $\triangleright$-equivariant.
\end{proposition}

\begin{proof}
First we prove 1) and 2). We have
\[
\begin{split}
(Y\triangleright c_{f,v}^{\Lambda})(X)
& = c_{f,v}^{\Lambda}(XY) = f(X\triangleright Y\triangleright v) = c_{f,Y\triangleright v}^{\Lambda}(X), \\
(c_{f,v}^{\Lambda}\triangleleft Y)(X)
& = c_{f,v}^{\Lambda}(YX) = f(Y\triangleright X\triangleright v) = c_{f\triangleleft Y,v}^{\Lambda}(X).
\end{split}
\]
To prove 3) we need to use the fact that $S$ is an anti-homomorphism. We have
\[
\begin{split}
(S(c_{f,v}^{\Lambda})\triangleleft Y)(X)
& = S(c_{f,v}^{\Lambda})(YX) = c_{f,v}^{\Lambda}(S(YX)) = f(S(X) \triangleright S(Y)\triangleright v) \\
&  = f(S(X) \triangleright (v \triangleleft Y)) = c_{f, v \triangleleft Y}^\Lambda (S(X)) = S(c_{f, v \triangleleft Y}^\Lambda) (X).
\end{split}
\]
The proof of 4) is similar to that of 3). We compute
\[
\begin{split}
(Y\triangleright S(c_{f,v}^{\Lambda}))(X) & = S(c_{f,v}^{\Lambda})(XY)=c_{f,v}^{\Lambda}(S(XY)) = f(S(Y)\triangleright S(X)\triangleright v) \\
& = (Y\triangleright f)(S(X)\triangleright v) = c_{Y\triangleright f,v}^{\Lambda}(S(X))
= S(c_{Y\triangleright f,v}^{\Lambda})(X).
\qedhere
\end{split}
\]
\end{proof}

These maps can be used to describe the action of $\Uqg$ on the matrix units $\mathsf{M}^n_m$ and $\mathsf{N}^n_m$.

\begin{corollary}
\label{cor:equiv-maps}
Let $\{v_m\}_m$ be an orthonormal basis of $V(\Lambda)$ and $\{f^n\}_n$ be the dual basis of $V(\Lambda)^*$.
We define the maps $\gamma_L^{(i, j)}, \gamma_R^{(i, j)}: V(\Lambda) \otimes V(\Lambda)^* \to \Cqg$ by the formulae
\[
\gamma_L^{(i, j)} (v_m \otimes f^n) := (\mathsf{N}^n_m)^i_j, \quad
\gamma_R^{(i, j)} (v_m \otimes f^n) := (\mathsf{M}^n_m)^i_j.
\]
Then $\gamma_L^{(i, j)}$ is $\triangleright$-equivariant and $\gamma_R^{(i, j)}$ is $\triangleleft$-equivariant.
\end{corollary}

\begin{proof}
The action of $\Uqg$ on $V(\Lambda) \otimes V(\Lambda)^*$ is the usual tensor product action, namely $X \triangleright (v \otimes f) = X_{(1)} \triangleright v \otimes X_{(2)} \triangleright f$. On the other hand on $\Cqg$ we have
\[
X \triangleright (\mathsf{N}^n_m)^i_j = (X_{(1)} \triangleright u^i_m) (X_{(2)} \triangleright u_n^{j*}) = (X_{(1)} \triangleright u^i_m) (X_{(2)} \triangleright S(u^n_j)),
\]
where the last step holds because we are considering orthonormal bases.
Since by definition we have $u^i_j = c^\Lambda_{f^i, v_j}$ the result follows from \cref{prop:equiv-maps}.

For the right action we similarly observe that
\[
(\mathsf{M}^n_m)^i_j \triangleleft X = (u^{m*}_i \triangleleft X_{(1)}) (u^n_j \triangleleft X_{(2)}) = (S(u^i_m) \triangleleft X_{(1)}) (u^n_j \triangleleft X_{(2)}).
\]
Then the result follows again from \cref{prop:equiv-maps}.
\end{proof}

\subsection{Generalized flag manifolds}

We follow the setup of \cite{quantum-flag}.
Let $S$ be a subset of the simple roots of $\mathfrak{g}$.
Then the \emph{quantized Levi factor} is defined as
\[
\Uqlevi = \textrm{algebra generated by } \{K_\lambda, E_i, F_i : i \in S\} \subset \Uqg.
\]
It is clear from the definition that $\Uqlevi$ is a Hopf $*$-subalgebra of $\Uqg$. Corresponding to the choice of $S$, the quantized coordinate rings of \emph{generalized flag manifolds} are defined as
\[
\begin{split}
\flagleviL & = \{a \in \mathbb{C}_q[G] : X \triangleright a = \varepsilon(X) a, \ \forall X \in \Uqlevi \}, \\
\flagleviR & = \{a \in \mathbb{C}_q[G] : a \triangleleft X = \varepsilon(X) a, \ \forall X \in \Uqlevi \}.
\end{split}
\]
It is easy to see that they are $*$-subalgebras of $\Cqg$.
The case of full flag manifolds corresponds to the choice $S = \emptyset$.
As in that case, we have right and left actions of $\Uqg$.

The aim is to apply the results of the previous sections to the case of generalized flag manifolds.
In order to do this we need to define appropriate matrices over $\flagleviL$ and $\flagleviR$ in terms of the matrix units $\mathsf{N}^n_m$ and $\mathsf{M}^n_m$.
The next result shows that it is equivalent to having a $U_{q}(\mathfrak{l}_S)$-invariant vector in $V(\Lambda) \otimes V(\Lambda)^*$.

\begin{proposition}
\label{prop:inv-vector}
Let $\mathsf{P} = \sum_{m, n} c^m_n \mathsf{M}^n_m$ and $\mathsf{Q} = \sum_{m, n} c^m_n \mathsf{N}^n_m$.
Define
\[
w = \sum_{m, n} c^m_n v_m \otimes f^n \in V(\Lambda) \otimes V(\Lambda)^*.
\]
Then $\mathsf{P}^i_j \in \flagleviR$ and $\mathsf{Q}^i_j \in \flagleviL$ if and only if $w$ is a $U_{q}(\mathfrak{l}_S)$-invariant vector.
\end{proposition}

\begin{proof}
We will spell the proof only for $\mathsf{Q}$, the other case is very similar.
Using the map $\gamma_{L}^{(i, j)} : V(\Lambda) \otimes V(\Lambda)^* \to \Cqg$ from \cref{cor:equiv-maps} we have the equality $\mathsf{Q}^i_j = \gamma_{L}^{(i, j)}(w)$.
This map is $\triangleright$-equivariant. Hence for any $X \in \Uqlevi$ we have
\[
X \triangleright \mathsf{Q}^i_j = X \triangleright \gamma_{L}^{(i, j)}(w) = \gamma_{L}^{(i, j)}(X \triangleright w).
\]
It is clear that if $X \triangleright w = \varepsilon(X) w$ then $X \triangleright \mathsf{Q}^i_j = \varepsilon(X) \mathsf{Q}^i_j$.

Conversely suppose that $X \triangleright \mathsf{Q}^i_j = \varepsilon(X) \mathsf{Q}^i_j$.
Then we must have $\gamma_{L}^{(i, j)}(X \triangleright w - \varepsilon(X) w) = 0$.
To prove that this implies $X \triangleright w = \varepsilon(X) w$, it suffices to show that if $\gamma_{L}^{(i, j)}(z) = 0$ for all $i, j$ then $z = 0$.
Write $z = \sum_{m, n} b^m_n v_m \otimes f^n$. The condition $\gamma_{L}^{(i, j)}(z) = 0$ is equivalent to $\sum_{m, n} b^m_n (\mathsf{N}^n_m)^i_j = 0$.
Since by \cref{prop:mat-unit1} we know that the matrices $\mathsf{N}^n_m$ are linearly independent we must have $b^m_n = 0$, hence $z = 0$.
\end{proof}

\begin{remark}
The module $V(\Lambda) \otimes V(\Lambda)^*$ always contains an invariant vector, corresponding to the trivial subrepresentation.
However this is not interesting for our purposes: indeed this vector is invariant under the whole $\Uqg$ and, as a consequence, the elements $\mathsf{P}$ and $\mathsf{Q}$
constructed in this way are multiples of the identity.
\end{remark}

\begin{remark}
It can be shown that, if $w \in V(\Lambda) \otimes V(\Lambda)^*$ is $\Uqlevi$-invariant with respect to the left action, then it also invariant with respect to the right action.
\end{remark}

The upshot is that, given a non-trivial invariant vector in $V(\Lambda) \otimes V(\Lambda)^*$, we can construct appropriate invariant matrices in terms of the matrix units $\mathsf{M}^n_m$ and $\mathsf{N}^n_m$.
However recall that for the construction of twisted $2$-cycles we need invariant projections. This leads to more complicated conditions on the invariant vector.
We will not attempt to discuss this problem in full generality, but rather present an interesting example in the next subsection.

\subsection{Quantum Grassmannians}

As an example of the setup discussed above, we will consider the quantum Grassmannians.
The quantized coordinate rings $\mathbb{C}_q[\mathrm{Gr}(r, N)]$ are defined by taking $\mathfrak{g} = \mathfrak{sl}(N)$ and $S$ to be the set of simple roots with $\alpha_r$ removed.

For our construction of invariant matrices we will pick $\Lambda = \omega_1$, corresponding to the fundamental representation. This representation can be realized as follows.

\begin{lemma}
The fundamental representation $V(\omega_1)$ of $U_{q}(\mathfrak{sl}(N))$ is realized on $\mathbb{C}^N$ by
\[
\pi(K_{k})v_{i} = q^{\delta_{i,k}-\delta_{i,k+1}}v_{i}, \quad
\pi(E_{k})v_{i} = \delta_{i}^{k+1}q^{-1/2}v_{i-1}, \quad
\pi(F_{k})v_{i} = \delta_{i}^{k}q^{1/2}v_{i+1}.
\]
The highest weight vector is given by $v_1$. Moreover this representation is unitary with respect to the standard Hermitian inner product on $\mathbb{C}^N$.
\end{lemma}

\begin{proof}
This follows from simple computations that we omit.
\end{proof}

Now we look for non-trivial $\Uqlevi$-invariant vectors in the tensor product $V(\omega_1) \otimes V(\omega_1)^*$, as in \cref{prop:inv-vector}.
We have $V(\omega_1) \otimes V(\omega_1)^* \cong V(0) \oplus V(\omega_1 + \omega_{N - 1})$ and classically the adjoint representation $V(\omega_1 + \omega_{N - 1})$ contains such an invariant vector.

\begin{lemma}
Let $w = \sum_{m = 1}^r v_m \otimes f^m \in V(\omega_1) \otimes V(\omega_1)^*$.
Then $w$ is $\Uqlevi$-invariant.
\end{lemma}

\begin{proof}
First of all recall the action on the dual, given by $X \triangleright f^i = \sum_j \pi(S(X))^i_j f^j$.
A simple computation then shows that $E_k \triangleright f^i = - \delta^i_k q^{1/2} f^{i + 1}$.
Then we can compute
\[
\begin{split}
E_k \triangleright (v_m \otimes f^m)
& = E_k \triangleright v_m \otimes K_k \triangleright f^m + v_m \otimes E_k \triangleright f^m \\
& = q^{-\delta_{m,k} + \delta_{m,k+1}} \delta_{m}^{k+1}q^{-1/2} v_{m-1} \otimes f^{m} - \delta_{m}^{k} q^{1/2} v_{m} \otimes f^{m+1}\\
& = \delta^{k+1}_m q^{1/2} v_{m - 1} \otimes f^m - \delta^k_m q^{1/2} v_m \otimes f^{m + 1}.
\end{split}
\]
Now we have to show that $E_k \triangleright w = 0$ for $k \neq r$.
This is clear for $k > r$, since the sum in $w$ runs from $1$ to $r$.
For $k < r$ on the other hand we have
\[
E_k \triangleright w = \sum_{m = 1}^r E_k \triangleright (v_m \otimes f^m) = q^{1/2} v_{k} \otimes f^{k + 1} - q^{1/2} v_k \otimes f^{k + 1} = 0.
\]
The computation showing invariance under $F_k$ is very similar and we omit it.
Moreover similar computations also show that $w$ is invariant with respect to the right action $\triangleleft$.
\end{proof}

Corresponding to this invariant vector, we get elements $\mathsf{P} = \sum_{m = 1}^r \mathsf{M}^m_m$ and $\mathsf{Q} = \sum_{m = 1}^r \mathsf{N}^m_m$.
It is clear that they are projections.
We will only consider $\mathsf{P}$ in the following.

\begin{lemma}
We have the relations
\[
\mathsf{P}^* = \mathsf{P}, \quad
\mathsf{P}^2 = \mathsf{P}, \quad
\mathrm{Tr}(K_{2 \rho}^{-1} \mathsf{P}) = q^{r - N} [r]_q.
\]
\end{lemma}

\begin{proof}
The first two relations follow from the general properties of the matrix units $\mathsf{M}^n_m$, while the last relation requires some extra computations. Recall from \cref{prop:mat-unit1} that $\mathrm{Tr}(K_{2 \rho}^{-1} \mathsf{P}) = \sum_{m = 1}^r q^{-(2 \rho, \lambda_m)}$.
The weights of the fundamental representation $V(\omega_1)$ are given by $\lambda_i = \omega_i - \omega_{i - 1}$ with $i = 1, \cdots, N$, where we use the convention $\omega_0 = \omega_N = 0$.
We also have the identity $2 \rho = \sum_{k = 1}^{N - 1} k (N - k) \alpha_k$.
Then it is easy to show that $(2\rho, \lambda_m) = N - 2m + 1$.
Finally a simple computation shows that $\sum_{m = 1}^r q^{-(2 \rho, \lambda_m)} = q^{r - N} [r]_q$.
\end{proof}

\begin{remark}
The entries of $\mathsf{P}$ actually generate the algebra $\mathbb{C}_q[\mathrm{Gr}(r, N)]$, as shown in \cite{kolb}.
This is reasonable, since for $q \to 1$ the above conditions mean that $\mathsf{P}$ is an orthogonal projection of rank $r$ and classically $\mathrm{Gr}(r, N)$ can be identified with the space of such matrices.
\end{remark}

Finally we show that the class of the twisted $2$-cycle $C(\mathsf{P})$ is non-trivial.

\begin{proposition}
The class $[C(\mathsf{P})] \in H_2^\theta(\mathbb{C}_q[\mathrm{Gr}(r, N)])$ is non-trivial.
\end{proposition}

\begin{proof}
We will use the first criterion in \cref{thm:conditions}.
For the projection $\mathsf{P} = \sum_{m = 1}^r \mathsf{M}^m_m$ we have $\chi_a(\mathsf{P}) = \sum_{i = 1}^r q^{(\alpha_a - 2\rho, \lambda_i)} [(\alpha_a, \lambda_i)]_q$.
We take $a = r$, where $r$ is the parameter defining the Grassmannian.
Since $\lambda_i = \omega_i - \omega_{i - 1}$ we get $(\alpha_r, \lambda_i) = \delta_{r, i} - \delta_{r, i - 1}$ and
\[
\chi_r(\mathsf{P}) = \sum_{i = 1}^r q^{(\alpha_r - 2\rho, \lambda_i)} [(\alpha_r, \lambda_i)]_q = q^{-(2\rho, \lambda_r) + 1}.
\]
This is non-zero and hence the class is non-trivial.
\end{proof}

In the classical limit $q \to 1$ the class $[C(\mathsf{P})]$ can be identified with a differential $2$-form, thanks to the Hochschild-Kostant-Rosenberg theorem.
In particular we can look at the case of projective spaces. Then it is possible to show that the class $[C(\mathsf{P})]$ corresponds, up to a scalar, with the Kähler form coming from the Fubini-Study metric.

\vspace{3mm}

{\footnotesize
\emph{Acknowledgements}.
I want to thank Adam Rennie for discussions on quantum projective spaces, which eventually led to this project.
Many thanks are also due to Robert Yuncken, whose comments have considerably improved this paper.
I also want to thank Stefan Kolb for answering a question on quantum Grassmannians.
I want to thank IISER in Trivandrum, India for granting me space to work on this project during the summer I spent there. The rest of this work was done at Université Clermont Auvergne in Clermont-Ferrand, France. Partially supported by the FWO grant G.0251.15N from Vrije Universiteit Brussel (VUB).
}

\bigskip

\end{document}